\documentclass[11pt, reqno]{amsart}
\usepackage{amsmath,amsopn,amssymb,amsthm,multicol}
\usepackage[cp1251]{inputenc}
\usepackage[english]{babel}
\usepackage[dvips]{graphicx}
\usepackage[breaklinks=true,colorlinks=true,linkcolor=blue,citecolor=blue,urlcolor=blue]{hyperref}

\renewenvironment{proof}[1][Proof]{\textbf{#1.} }
{\ \rule{0.5em}{0.5em}}

\DeclareMathOperator{\ad}{ad}
\DeclareMathOperator{\Ad}{Ad}
\DeclareMathOperator{\diag}{diag}

\DeclareMathOperator{\trace}{trace}

\DeclareMathOperator{\Id}{Id}
\DeclareMathOperator{\End}{End}
\DeclareMathOperator{\Aut}{Aut}

\renewcommand{\arraystretch}{1.5}

\newtheorem{theorem}{Theorem}
\newtheorem{pred}{Proposition}
\newtheorem{lem}{Lemma}
\newtheorem{cor}{Corollary}

\newtheorem{remark}{Remark}

\begin{document}

\title[Classification of generalized Wallach spaces]
{Classification of generalized Wallach spaces}

\author{Yu.\,G.~Nikonorov}
\address{Yu.\,G. Nikonorov \newline
South Mathematical Institute of  Vladikavkaz Scientific Centre \newline
of the Russian Academy of Sciences, Vladikavkaz, Markus st. 22, \newline
362027, Russia}
\email{nikonorov2006@mail.ru}

%\thanks{}

\begin{abstract} In this paper, we present the classification of generalized  Wallach spaces and discuss
some related problems.

\vspace{2mm} \noindent Key word and phrases:
generalized Wallach space, compact homogeneous space, symmetric space,
automorphism of a Lie algebra, Killing form, Riemannian metric, Einstein metric, Ricci flow.

\vspace{2mm}

\end{abstract}

\maketitle

\section*{Introduction and the main result}

This paper is devoted to the classification of generalized Wallach spaces, a remarkable class of compact homogeneous spaces.
These spaces were introduced in the paper \cite{Nikonorov2}, where they were called three-symmetric spaces.
Now we prefer to call them {\it generalized Wallach spaces} as in \cite{Nikonorov1}, because this term is less confusing and more informative.
We begin with recalling some notations and definitions.

Let $G/H$ be a compact homogeneous spaces with connected  compact semisimple Lie group $G$ and its compact subgroup $H$.
Denote by $\mathfrak{g}$ and $\mathfrak{h}$ Lie algebras of $G$ and $H$ respectively.
We suppose that $G/H$ is almost effective, i.~e. there is no non-trivial ideals of the Lie algebra $\mathfrak{g}$  in $\mathfrak{h} \subset\mathfrak{g}$.
Denote by $B=B(\boldsymbol{\cdot}\,,\boldsymbol{\cdot})$ the Killing form of $\mathfrak{g}$. Since $G$ is compact, $B$ is negatively defined on $\mathfrak{g}$.
Therefore, $\langle\boldsymbol{\cdot}\,,\boldsymbol{\cdot}\rangle:=-B(\boldsymbol{\cdot}\,,\boldsymbol{\cdot})$ is a positive definite inner product
on $\mathfrak{g}$.
Properties of $B$ imply that
$\langle\boldsymbol{\cdot}\,,\boldsymbol{\cdot}\rangle$ is bi-invariant, i.~e. $\langle[Z,X],Y\rangle+\langle X,[Z,Y]\rangle=0$ for all $X,Y,Z \in \mathfrak{g}$.

Let $\mathfrak{p}$ be the $\langle\boldsymbol{\cdot}\,,\boldsymbol{\cdot}\rangle$-orthogonal complement to $\mathfrak{h}$ in $\mathfrak{g}$.
It is clear that $\mathfrak{p}$ is $\Ad(H)$-invariant (and $\ad(\mathfrak{h})$-invariant, in particular).
The module $\mathfrak{p}$ is naturally identified with the tangent space to $G/H$ at the point $eH$, see e.~g. \cite[7.23]{Bes}.
Every $G$-invariant Riemannian metric on~$G/H$ generates an~$\Ad(H)$-invariant
inner product on~$\mathfrak{p}$ and vice versa \cite[7.24]{Bes}. Therefore, it is possible to
identify invariant Riemannian metrics on~$G/H$ with $\Ad(H)$-invariant inner
products on~$\mathfrak{p}$. Note that the Riemannian metric generated by the~inner product
$\langle\boldsymbol{\cdot}\,,\boldsymbol{\cdot}\rangle\bigr\vert_{\mathfrak{p}}$ is called
{\it standard} or {\it Killing}.

\begin{remark}
A linear subspace $\mathfrak{q}\subset \mathfrak{p}$ is $\ad(\mathfrak{h})$-invariant if and only if it is $\Ad(H_0)$-invariant,
where $H_0$ is the unit component of the group $H$. Hence,  these two notions are equivalent for connected~$H$.
It should be noted also, that the group $H$ is connected provided that the space $G/H$ is simply connected.
\end{remark}

\smallskip

Suppose that a homogeneous space~$G/H$ has the following property: the modules $\mathfrak{p}$
is decomposed as a direct sum of three
$\Ad(H)$-invariant irreducible modules pairwise orthogonal
with respect to~$\langle\boldsymbol{\cdot}\,,\boldsymbol{\cdot}\rangle$, i.~e.
\begin{equation}\label{decom1}
\mathfrak{p}=\mathfrak{p}_1\oplus \mathfrak{p}_2\oplus \mathfrak{p}_3,
\end{equation}
such that
\begin{equation}\label{decom2}
[\mathfrak{p}_i,\mathfrak{p}_i]\subset \mathfrak{h} \quad \mbox{ for } \quad i\in\{1,2,3\}.
\end{equation}
Homogeneous spaces with this property are called {\it generalized Wallach spaces}.

\begin{remark} The authors of  \cite{Nikonorov2, Lomshakov2} called these spaces {\it three-locally-symmetric}, since the  condition
{\rm(\ref{decom2})} resembles the~condition of local symmetry for
homogeneous spaces \big(a~locally symmetric homogeneous space~$G/H$ is
characterized by the~relation $[\mathfrak{p},\mathfrak{p}]\subset \mathfrak{h}$,
where $\mathfrak{g}=\mathfrak{h}\oplus \mathfrak{p}$
and~$\mathfrak{p}$ is $\Ad(H)$-invariant~\cite[7.70]{Bes}\big).
\end{remark}

A detailed discussion on generalized Wallach spaces could be found in \cite[pp. 6346--6347]{Nikonorov1} or~\cite{Lomshakov2},
but we recall some important properties of these spaces for the  reader's convenience.

There are many examples of these spaces, e.~g. the~manifolds of complete flags in the complex, quaternionic, and Cayley projective planes
(a complete flag in any of these planes is a pair
$(p,l)$ where $p$ is a point in the plane and $l$ a  line (complex, quaternionic or octonionic) containing the point $p$):
$$
SU(3)/T_{\max},
\ \ Sp(3)/Sp(1)\times Sp(1)\times Sp(1),
\ \ F_4/Spin(8).
$$
These spaces (known as {\it Wallach spaces}) are  also interesting in that
they admit invariant Riemannian metrics of positive sectional curvature
(see~\cite{Wal}).
The~Lie group~$SU(2)$ \big($H=\{e\}$\big) is another example
of generalized Wallach spaces.
Note also that $SO(3)/(\mathbb{Z}_2\times \mathbb{Z}_2)$ is
the~manifold of complete flags in the real projective plane.
It is interesting that the above manifolds of complete flags have representations as so-called {\it Cartan’s isoparametric submanifolds},
see e.~g. \cite{Sanchez} for details.

Other examples of generalized Wallach spaces are some K\"ahler $C$-spaces such as
$$
SU(n_1+n_2+n_3)\big/S\big(U(n_1)\times U(n_2)\times U(n_3)\big),
$$
$$
SO(2n)/U(1)\times U(n-1),
\quad
E_6/U(1)\times U(1)\times Spin(8).
$$
There are two more $3$\,-parameter families of generalized Wallach spaces:
$$
SO(n_1+n_2+n_3)\big/SO(n_1)\times SO(n_2)\times SO(n_3), \quad Sp(n_1+n_2+n_3)\big/Sp(n_1)\times Sp(n_2)\times Sp(n_3).
$$

Note, that every generalized Wallach space admits a $3$\,-parameter family of invariant Riemannian metrics
determined by $\Ad (H)$-invariant inner products
\begin{equation}
\label{metric}
(\boldsymbol{\cdot}\,,\boldsymbol{\cdot})=\left.x_1\,\langle\boldsymbol{\cdot}\,,\boldsymbol{\cdot}\rangle\right|_{\mathfrak{p}_1}+
\left.x_2\,\langle\boldsymbol{\cdot}\,,\boldsymbol{\cdot}\rangle\right|_{\mathfrak{p}_2}+
\left.x_3\,\langle\boldsymbol{\cdot}\,,\boldsymbol{\cdot}\rangle\right|_{\mathfrak{p}_3},
\end{equation}
where $x_1$, $x_2$, $x_3$ are positive real numbers.

In~\cite{Nikonorov2},  it was shown that every generalized Wallach space
admits at least one invariant Einstein metric.
This result could not be improve in general (e.~g.~$SU(2)$ admits exactly one invariant Einstein metric).
Later in \cite{Lomshakov2}, a detailed study of invariant Einstein metrics was developed
for all generalized Wallach spaces. In particular, it is proved that there are at most four
Einstein metrics (up to homothety) for every such space.
A detailed discussion and the references related to all known results on Einstein invariant metrics on generalized Wallach spaces
one can find in~\cite{Nikonorov1}. More detailed information on invariant Einstein metric on general homogeneous spaces could be found in
\cite{Bes, Boehm2004, BWZ, Wang1, Wang2}.

In the recent papers \cite{AANS1,AANS2}, generalized Wallach spaces were studied from the point of view of the Ricci flow.
Some results of these papers we will discuss in the last section.
\smallskip

Denote by $d_i$ the~dimension of~$\mathfrak{p}_i$. Let~$\big\{e^j_i\big\}$ be
an~orthonormal basis in~$\mathfrak{p}_i$ with respect to
$\langle\boldsymbol{\cdot}\,,\boldsymbol{\cdot}\rangle$, where
$i\in\{1,2,3\}$, $1\le j\le d_i=\dim(\mathfrak{p}_i)$. Consider
the~expression $[ijk]$ defined by
the~equality
\begin{equation}\label{constwz}
[ijk]=\sum_{\alpha,\beta,\gamma}
\left\langle\big[e_i^\alpha,e_j^\beta \big],e_k^\gamma \right\rangle^2,
\end{equation}
where $\alpha$, $\beta$, and $\gamma$ range from~1 to ~$d_i$, $d_j$, and
$d_k$ respectively, see \cite{WZ3}. The~symbols $[ijk]$ are symmetric in all three indices by the bi-invariance
of the~metric
$\langle\boldsymbol{\cdot}\,,\boldsymbol{\cdot}\rangle$.
Moreover, for spaces under consideration, we have
$[ijk]=0$
if two indices coincide. Therefore, the~quantity
\begin{equation}\label{consta}
A:=[123]
\end{equation}
plays an~important role.
It easy to see that $d_i\ge 2A$ for all $i=1,2,3$ (see \cite{Nikonorov2} or Lemma \ref{calca} below).
Hence the following constant
\begin{equation}\label{a123}
a_i=\frac{A}{d_i},\quad i\in \{1,2,3\},
\end{equation}
are such that $(a_1,a_2,a_3)\in [0,1/2]^3$.
Note, that these constants completely determine some important properties of a generalized Wallach space $G/H$, e.~g. the equation
of the Ricci flow on~$G/H$, see \cite{AANS1,AANS2}. Of course, not every triple
$(a_1,a_2,a_3)\in [0,1/2]^3$ corresponds to some generalized Wallach space.
A complete description of suitable triples we will get together with the classification of generalized Wallach spaces.

\smallskip
Now we are ready to state the main result of this paper.

\begin{theorem}\label{main}
Let $G/H$ be a connected and simply connected compact homogeneous space. Then $G/H$ is a generalized Wallach space if and only if
one of the following assertions holds:

{\rm 1)} $G/H$ is a direct product of three irreducible symmetric spaces of compact type {\rm (}$A=a_1=a_2=a_3=0$ in this case{\rm)};

{\rm 2)} The group $G$ is simple and the pair $(\mathfrak{g}, \mathfrak{h})$ is one of the pairs in Table 1 {\rm (}the embedding of $\mathfrak{h}$
to $\mathfrak{g}$ is determined by the following requirement: the corresponding pairs $(\mathfrak{g}, \mathfrak{k}_i)$ and $(\mathfrak{k}_i,\mathfrak{h})$,
$i=1,2,3$, in Table 2 are symmetric{\rm )};

{\rm 3)} $G=F\times F\times F \times F$ and $H=\diag(F)\subset G$ for some connected simply connected compact simple Lie group $F$,
with the following description on the Lie algebra level:
$$
(\mathfrak{g}, \mathfrak{h})= \bigl(\mathfrak{f}\oplus \mathfrak{f}\oplus\mathfrak{f}\oplus\mathfrak{f},\,\diag(\mathfrak{f})=\{(X,X,X,X)\,|\, X\in f\}\bigr),
$$
where $\mathfrak{f}$ is the Lie algebra of $F$, and {\rm (}up to permutation{\rm)}
$\mathfrak{p}_1=\{(X,X,-X,-X)\,|\, X\in f\}$, \linebreak
$\mathfrak{p}_2\!=\!\{(X,-X,X,-X)\,|\, X\in f\}$,\!
$\mathfrak{p}_3\!=\!\{(X,-X,-X,X)\,|\, X\in f\}$\! {\rm(}$a_1\!=\!a_2\!=\!a_3\!=\!1/4$ in this case{\rm)}.
\end{theorem}

{\small
\renewcommand{\arraystretch}{1.8}
\begin{table}[t]
{\bf Table 1.} The pairs $(\mathfrak{g},\mathfrak{h})$ corresponded to generalized Wallach spaces $G/H$ with simple $G$.\newline
\begin{center}
\begin{tabular}{|c|c|c|c|c|c|c|c|c|}
\hline
N& \ $\mathfrak{g}$& $\mathfrak{h}$ & $d_1$ &  $d_2$ & $d_3$  & $a_1$ &  $a_2$ & $a_3$ \\
\hline\hline
1&\!\!$so(k+l+m)$\!\!\!&\!\!$so(k)\oplus so(l) \oplus so(m)$\!\!& $kl$ & $km$ & $lm$ &
\!\!$\frac{m}{2(k+l+m-2)}$\!\!&\!\!$\frac{l}{2(k+l+m-2)}$\!\!&\!\!$\frac{k}{2(k+l+m-2)}$\!\!\\  \hline
2&\!\!$su(k+l+m)$\!\!\!&$\!\!s(u(k)\oplus u(l) \oplus u(m))$\!\!& $2kl$ & $2km$ & $2lm$ &
\!\!$\frac{m}{2(k+l+m)}$\!\!&\!\!$\frac{l}{2(k+l+m)}$\!\!&\!\!$\frac{k}{2(k+l+m)}$\!\!\\  \hline
3&\!\!$sp(k+l+m)$\!\!\!&\!\!$sp(k)\oplus sp(l) \oplus sp(m)$\!\!& $4kl$ & $4km$ & $4lm$ &
\!\!$\frac{m}{2(k+l+m+1)}$\!\!&\!\!$\frac{l}{2(k+l+m+1)}$\!\!&\!\!$\frac{k}{2(k+l+m+1)}$\!\!\\  \hline

4&$su(2l)$,\,\, $l\geq 2$ &$u(l)$ &\!\!$l(l-1\!)$\!\! &\!\!$l(l+1\!)$\!\!&\!\!$l^2-1$\!\!&  $\frac{l+1}{4l}$ &
$\frac{l-1}{4l}$ &  $\frac{1}{4}$ \\  \hline

5&$so(2l)$,\,\, $l\geq 4$ &$u(1)\oplus u(l-1\!)$ &\!\!$2(l-1\!)$\!\! &\!\!$2(l-1\!)$\!\!&\!\!$(l\!-\!1\!)(l\!-\!2)$\!\!&  $\frac{l-2}{4(l-1)}$ &
$\frac{l-2}{4(l-1)}$ &  $\frac{1}{2(l-1)}$ \\  \hline
6&$e_6$ &$su(4)\oplus 2sp(1)\oplus \mathbb{R}$   &$16$ & $16$ & $24$ &  $\frac{1}{4}$ &  $\frac{1}{4}$ &  $\frac{1}{6}$ \\  \hline
7&$e_6$ &$so(8)\oplus \mathbb{R}^2$ & $16$ & $16$ & $16$ &  $\frac{1}{6}$ &  $\frac{1}{6}$ &  $\frac{1}{6}$ \\  \hline
8&$e_6$ &$sp(3)\oplus sp(1)$ & $14$ & $28$ & $12$ &  $\frac{1}{4}$ &  $\frac{1}{8}$ &  $\frac{7}{24}$ \\  \hline

9&$e_7$ &$so(8)\oplus 3sp(1)$ & $32$ & $32$ & $32$ &  $\frac{2}{9}$ &  $\frac{2}{9}$ &  $\frac{2}{9}$ \\  \hline
10&$e_7$ &$su(6)\oplus sp(1)\oplus \mathbb{R}$ & $30$ & $40$ & $24$ &  $\frac{2}{9}$ &  $\frac{1}{6}$ &  $\frac{5}{18}$ \\  \hline
11&$e_7$ &$so(8)$ & $35$ & $35$ & $35$ &  $\frac{5}{18}$ &  $\frac{5}{18}$ &  $\frac{5}{18}$ \\  \hline

12&$e_8$ &$so(12)\oplus 2sp(1)$ & $64$ & $64$ & $48$ &  $\frac{1}{5}$ &  $\frac{1}{5}$ &  $\frac{4}{15}$ \\  \hline
13&$e_8$ &$so(8)\oplus so(8)$ & $64$ & $64$ & $64$ &  $\frac{4}{15}$ &  $\frac{4}{15}$ &  $\frac{4}{15}$ \\  \hline

14&$f_4$ &$so(5)\oplus 2sp(1)$ & $8$ & $8$ & $20$ &  $\frac{5}{18}$ &  $\frac{5}{18}$ &  $\frac{1}{9}$ \\  \hline
15&$f_4$ &$so(8)$ & $8$ & $8$ & $8$ &  $\frac{1}{9}$ &  $\frac{1}{9}$ &  $\frac{1}{9}$ \\  \hline

\end{tabular}
\end{center}
\end{table}
\renewcommand{\arraystretch}{1}
}

The paper is divided into five sections. In Section 1 we discuss connections between
generalized Wallach spaces and $\mathbb{Z}_2\times\mathbb{Z}_2$-subgroups
in the automorphism groups $\Aut (\mathfrak{g})$ of compact Lie algebras~$\mathfrak{g}$.
The second section is devoted to general structural results on generalized Wallach spaces $G/H$ with connected $H$.
In Section 3 we get a classification of generalized Wallach spaces $G/H$ with simple $G$ and connected $H$.
In Section 4 we calculate the values of $a_1$, $a_2$, $a_3$ for all pairs in Table~1.
Finally, in the last section we discuss properties of the set of points $(a_1,a_2,a_3)\in [0,1/2]^3 \subset \mathbb{R}^3$.
\medskip

The {\bf proof of Theorem \ref{main}} follows immediately from Theorem~\ref{struc8}, Theorem~\ref{simplegroup}, and Proposition~\ref{struc9}.
The calculations of $a_1$, $a_2$, $a_3$ for all pairs in Table 1 are performed in Section \ref{calculation}.
\bigskip

The author thanks  Christoph B\"{o}hm, Vicente Cort\'{e}s, and Yuri Nikolayevsky  for interesting discussions concerning this project.

\bigskip

\section{Generalized Wallach spaces and involutive automorphisms}\label{waau}

Let us consider connected compact homogeneous spaces $G/H$
with the properties (\ref{decom1}) and (\ref{decom2}). We emphasize that we do not demand that the modules $\mathfrak{p}_i$ are $\Ad(H)$-irreducible now.
The inclusion $[\mathfrak{p}_i,\mathfrak{p}_i]\subset \mathfrak{h}$ implies that
\begin{equation}\label{subalgi}
\mathfrak{k}_i:=\mathfrak{h}\oplus \mathfrak{p}_i
\end{equation}
is a~subalgebra of~$\mathfrak{g}$ for any $i$, and the pair $(\mathfrak{k}_i, \mathfrak{h})$ is irreducible symmetric (it could be non-effective, of course).
From (\ref{decom1}) and (\ref{decom2}) we easily get that $[\mathfrak{p}_j, \mathfrak{p}_k]\subset \mathfrak{p}_i$ for pairwise distinct $i,j,k$.
Therefore,
$$
[\mathfrak{p}_j\oplus \mathfrak{p}_k, \mathfrak{p}_j\oplus \mathfrak{p}_k]\subset \mathfrak{h}\oplus \mathfrak{p}_i=\mathfrak{k}_i, \quad \{i,j,k\}=\{1,2,3\},
$$
and all the pairs $(\mathfrak{g},\mathfrak{k}_i)$ are also irreducible symmetric.

Let us consider involutive automorphisms
$$
\sigma_i: \mathfrak{g} \mapsto \mathfrak{g}, \quad i \in \{1,2,3\},
$$
of the Lie algebra $\mathfrak{g}$, such that
$$
\sigma_i|_{\mathfrak{k}_i} = \Id, \quad \sigma_i|_{\mathfrak{p}_j\oplus \,\mathfrak{p}_k} = - \Id,
$$
which do exist due to well known structure results (see e.~g. \cite[theorem 8.1.4]{Wolf2011}).
It is easy to see that
$$
\sigma_i \circ \sigma_j =\sigma_j \circ \sigma_i =\sigma_k
$$
for pairwise distinct $i,j,k$. Keeping in mind, that $\sigma_1 \circ \sigma_1 = \sigma_2 \circ \sigma_2=\sigma_3 \circ \sigma_3=\Id$ on $\mathfrak{g}$,
we get the following

\begin{pred}\label{sruct1}
The automorphisms $\sigma_1$,  $\sigma_2$, and $\sigma_3$ generate
a $\mathbb{Z}_2\times\mathbb{Z}_2$-subgroup in $\Aut (\mathfrak{g})$, the group of automorphisms of the Lie algebra $\mathfrak{g}$.
Every pair of these automorphisms are the generators of this group.
\end{pred}

Now, let $\mathfrak{g}$ be a compact semisimple Lie algebra and let $\Gamma$ be a $\mathbb{Z}_2\times\mathbb{Z}_2$-subgroup
in the group of automorphisms $\Aut (\mathfrak{g})$ of $\mathfrak{g}$. Suppose that $\sigma_1$ and $\sigma_2$ are generators of  $\Gamma$,
and consider  an inner product $-B$ on $\mathfrak{g}$, where $B$ is the Killing form of $\mathfrak{g}$.
Since $\sigma_1\circ \sigma_1=\sigma_2\circ \sigma_2=\Id$ and $\sigma_1\circ \sigma_2=\sigma_2\circ \sigma_1$,
we have {\bf commutating normal operators $\sigma_1$ and $\sigma_2$} on the Euclidean space $(\mathfrak{g},-B)$.
Moreover, since they are involutions, their eigenvalues are exactly $1$ and $-1$.
Therefore, these operator could be {\bf diagonalized simultaneously}, see e.~g. \cite[2.5.15]{HJ2013}.
\smallskip

Let consider the following linear subspaces of $\mathfrak{g}$:
\begin{eqnarray*}
\mathfrak{h}=\{X\in \mathfrak{g}\,|\, \sigma_1(X)=\sigma_2(X)=X\}, &
\mathfrak{p}_1=\{X\in \mathfrak{g}\,|\, -\sigma_1(X)=\sigma_2(X)=X\},\\
\mathfrak{p}_2=\{X\in \mathfrak{g}\,|\, \sigma_1(X)=-\sigma_2(X)=X\}, &
\mathfrak{p}_3=\{X\in \mathfrak{g}\,|\, -\sigma_1(X)=-\sigma_2(X)=X\}.\,\,\,\,
\end{eqnarray*}
Clear, that all this subspaces are pairwise orthogonal with respect to $-B$,
$\mathfrak{h}$ is a Lie subalgebra in $\mathfrak{g}$, $\mathfrak{g}=\mathfrak{h}\oplus \mathfrak{p}_1 \oplus \mathfrak{p}_2 \oplus \mathfrak{p}_3$,
and $[\mathfrak{p}_i,\mathfrak{p}_i]\subset \mathfrak{h}$, $i \in \{1,2,3\}$.
Therefore, we get a compact homogeneous space $G/H$ with the properties (\ref{decom1}) and (\ref{decom2}),
where $G$ is a connected and simply connected Lie group with the Lie algebra
$\mathfrak{g}$ and $H$ is its connected subgroup corresponding to the Lie subalgebra $\mathfrak{h}$.
Hence we get the following

\begin{theorem} \label{theorem1}
There is a one-to-one correspondence between
$\mathbb{Z}_2\times\mathbb{Z}_2$-subgroups
in the automorphism groups $\Aut (\mathfrak{g})$ of compact semisimple Lie algebras $\mathfrak{g}$ and
connected and simply connected compact homogeneous spaces
with the properties {\rm(\ref{decom1})} and {\rm(\ref{decom2})}.
\end{theorem}

In order to classify all (connected and simply connected) generalized Wallach spaces, it is enough to classify ``suitable''
$\mathbb{Z}_2\times\mathbb{Z}_2$-subgroups in the group of automorphisms $\Aut (\mathfrak{g})$ of compact semisimple Lie algebras $\mathfrak{g}$.
Here, ``suitable'' means that the corresponding modules $\mathfrak{p}_i$ are $\Ad(H)$-irreducible or, equivalently (due to connectedness of $H$),
$\ad(\mathfrak{h})$-irreducible. We will realize this idea for generalized Wallach spaces $G/H$ with simple $G$ in Section \ref{simple}. But in the general
case we should get more detailed structural results in the next section.

\section{On the structure of generalized Wallach spaces}

Here we consider the structure of {\bf a generalized Wallach space $G/H$ with connected $H$}.
Recall, that the properties of a module $\mathfrak{q}\subset \mathfrak{p}$ to be $\Ad(H)$-invariant and
$\ad(\mathfrak{h})$-invariant are equivalent for a connected group $H$.
We will use notations as above.
Since the Lie algebra $\mathfrak{g}$ is semisimple, then we can decompose it into a
($\langle\boldsymbol{\cdot}\,,\boldsymbol{\cdot}\rangle$-orthogonal) sum of simple ideals
$$
\mathfrak{g}=\mathfrak{g}_1\oplus\mathfrak{g}_2\oplus \cdots \oplus \mathfrak{g}_s.
$$
Let $\varphi_i:\mathfrak{h} \rightarrow \mathfrak{g}_i$ be the $\langle\boldsymbol{\cdot}\,,\boldsymbol{\cdot}\rangle$-orthogonal projection.
It is easy to see that all these projections are Lie algebra homomorphisms.
We rearrange indices so that $\varphi_i(\mathfrak{h})\neq \mathfrak{g}_i$ for $i=1,2,\dots,p$ and $\varphi_i(\mathfrak{h})= \mathfrak{g}_i$ for $i=p+1,\dots,s$.

Since the Lie algebra $\mathfrak{h}$ is compact, then we can decompose it into a
($\langle\boldsymbol{\cdot}\,,\boldsymbol{\cdot}\rangle$-orthogonal) sum of the center and simple ideals
$$
\mathfrak{h}= \mathbb{R}^l\oplus \mathfrak{h}_1\oplus\mathfrak{h}_2\oplus \cdots \oplus \mathfrak{h}_m.
$$
For $i=1,\dots,m$, we denote by $a^i$ the vector $(a^i_1,a^i_2,\dots, a^i_s)\in \mathbb{R}^s$, where
$a^i_j=1$, if $\varphi_j(\mathfrak{h}_i)$ is isomorphic to $\mathfrak{h}_i$, and $a^i_j=0$, if $\varphi_j(\mathfrak{h}_i)$ is a trivial Lie algebra
(there is no another possibility, because $\varphi_j$ is a Lie algebra homomorphism).
It is easy to see that $\sum_{i=1}^m a^i_j=1$ for $j=p+1,\dots,s$, since $\varphi_j(\mathfrak{h})= \mathfrak{g}_j$ is a simple Lie algebra.
Denote also the number $\dim(\varphi_i(\mathbb{R}^l))$ by $u_i$ for $i=1,\dots,s$, and put $u=\sum_{i=1}^s u_i$, $v_i=\sum_{j=1}^s a^i_j$ for $i=1,\dots,m$.
It is clear that $u \geq l$ and $v_i\geq 1$ for all $i$.

\begin{lem}\label{struc1} In the above notation and suggestions, the following inequality holds:
$$
p+u+\sum_{i=1}^m v_i - l-m =p+(u-l)+\sum_{i=1}^m (v_i-1) \leq 3.
$$
\end{lem}

\begin{proof}
For $i=1,2,\dots,p$, every $\mathfrak{g}_i$ contains at least one irreducible modules $\mathfrak{p}_j \subset \mathfrak{p}$, since
$\varphi_i(\mathfrak{h})\neq \mathfrak{g}_i$ and $\langle\boldsymbol{\cdot}\,,\boldsymbol{\cdot}\rangle$-orthogonal complement to
$\varphi_i(\mathfrak{h})$ in $\mathfrak{g}_i$ is a subset of $\mathfrak{p}$. This gives at least $p$ irreducible modules.
Further, an $\langle\boldsymbol{\cdot}\,,\boldsymbol{\cdot}\rangle$-orthogonal complement to
$\mathbb{R}^l$ in $\oplus_{i=1}^s \varphi_i(\mathbb{R}^l)$ is also a subset of $\mathfrak{p}$. It is clear that $\ad(\mathfrak{h})$ acts trivially
on this complement, hence we get exactly $u-l$ one-dimensional irreducible submodules in it.
Finally, for any $i=1,\dots,m$, an $\langle\boldsymbol{\cdot}\,,\boldsymbol{\cdot}\rangle$-orthogonal complement to
$\mathfrak{h}_i$ in $\oplus_{j=1}^s \varphi_j(\mathfrak{h}_i)$ is also subset of $\mathfrak{p}$.
In fact, we deal with compliment to $\diag(\mathfrak{h}_i)$ in $\mathfrak{h}_i\oplus \mathfrak{h}_i\oplus \cdots \oplus \mathfrak{h}_i$
($v_i$ pairwise isomorphic summands).
In this case we have exactly $(v_i-1)$ $\ad(\mathfrak{h})$-irreducible modules. Summing all numbers of irreducible submodules, we get the lemma.
\end{proof}

Without loss of generality we may rearrange the indices so that $v_1 \geq v_2 \geq \cdots \geq v_{m-1}\geq v_m (\geq 1)$. Then we get the following

\begin{cor}\label{struc2} In the above notation and suggestions, the following inequality holds:
$$
p\leq 3, \quad u-l \leq 3, \quad v_4=1, \quad v_3\leq 2.
$$
\end{cor}

\begin{lem}\label{struc3} If two of the modules $\mathfrak{p}_1,\mathfrak{p}_2,\mathfrak{p}_3$ are subsets of $\mathfrak{g}_i$ for some $i=1,\dots s$, then
the third modules is also a subset of $\mathfrak{g}_i$. In this case $p=1$.
\end{lem}

\begin{proof}
Suppose, e.~g. that $\mathfrak{p}_1,\mathfrak{p}_2 \subset \mathfrak{g}_i$, then $[\mathfrak{p}_1,\mathfrak{p}_2]\subset \mathfrak{g}_i\cap \mathfrak{p}_3$.
If $[\mathfrak{p}_1,\mathfrak{p}_2]\neq 0$, then a nonempty module $\mathfrak{g}_i\cap \mathfrak{p}_3$ is $\ad(\mathfrak{h})$-invariant, since
this property have both $\mathfrak{p}_3$ and $\mathfrak{g}_i$ (as an ideal in $\mathfrak{g}$).
On the other hand, $\mathfrak{g}_i\cap \mathfrak{p}_3 \subset \mathfrak{p}_3$ and
$\mathfrak{p}_3$ is $\ad(\mathfrak{h})$-irreducible. Therefore, $\mathfrak{g}_i\cap \mathfrak{p}_3 = \mathfrak{p}_3$ and  $\mathfrak{p}_3 \subset \mathfrak{g}_i$.

If $\mathfrak{g}_i\neq \varphi_i(\mathfrak{h})+ \mathfrak{p}_1 + \mathfrak{p}_2$, then we get $\mathfrak{p}_3 \subset \mathfrak{g}_i$ again,
because an $\langle\boldsymbol{\cdot}\,,\boldsymbol{\cdot}\rangle$-orthogonal complement to $\varphi_i(\mathfrak{h})$ in $\mathfrak{g}_i$
is a subset of $\mathfrak{p}$.

Now, suppose that $[\mathfrak{p}_1,\mathfrak{p}_2]=0$ and $\mathfrak{g}_i= \varphi_i(\mathfrak{h})+ \mathfrak{p}_1 + \mathfrak{p}_2$.
Note that $[\varphi_i(\mathfrak{h}),\mathfrak{p}_1]\subset \mathfrak{p}_1$ ($[Y,X]=[\varphi_i(Y),X] \subset \mathfrak{p}_1$ for every
$Y\in \mathfrak{h}$ and every $X\in \mathfrak{p}_1 \subset \mathfrak{g}_i$), $[\mathfrak{p}_2,\mathfrak{p}_1]=0$, and
$[\mathfrak{p}_1,\mathfrak{p}_1]\subset \mathfrak{h}\cap \mathfrak{g}_i$,
hence $[\mathfrak{g}_i,\mathfrak{p}_1] \subset
[\mathfrak{p}_1,\mathfrak{p}_1]+\mathfrak{p}_1$ and  $[\mathfrak{g}_i,[\mathfrak{p}_1,\mathfrak{p}_1]] \subset
[\mathfrak{p}_1,[\mathfrak{p}_1,\mathfrak{p}_1]+ \mathfrak{p}_1]\subset [\mathfrak{p}_1,\mathfrak{p}_1]+ \mathfrak{p}_1$ (by the Jacoby equality).
Therefore, $[\mathfrak{p}_1,\mathfrak{p}_1]+ \mathfrak{p}_1$ is a proper ideal in $\mathfrak{g}_i$, that is impossible.

The last assertion of the lemma is obvious.
\end{proof}

\begin{lem}\label{struc4} If $A=0$, then $G/H$ is locally a direct product of three  irreducible symmetric spaces of compact type.
A simply connected $G/H$ with $A=0$ is a direct product of three  irreducible symmetric spaces of compact type.
\end{lem}

\begin{proof}
It is known that $A=0$ if and only if the space $G/H$ is locally a direct product of three compact irreducible symmetric spaces
(see  \cite[Theorem 2]{Lomshakov2}). Finally, we remind that complete (in particular, homogeneous) and simply connected locally symmetric space
is a symmetric space, see e.~g. \cite[Theorem 5.6]{Helgason1978}. Hence we get the lemma.
\end{proof}

\begin{cor}\label{struc5} If $p \geq 2$, then $A=0$, consequently, $G/H$ locally is a direct product of three  irreducible symmetric spaces of compact type.
\end{cor}

\begin{proof}
By Lemma \ref{struc3} we get that one of the modules  $\mathfrak{p}_1$, $\mathfrak{p}_2$, $\mathfrak{p}_3$  is in $\mathfrak{g}_1$ and the second
one is in~$\mathfrak{g}_2$. Hence, $[\mathfrak{p}_1,\mathfrak{p}_2]=0$ and $A=0$. Now, it suffices to apply Lemma \ref{struc4}.
\end{proof}

\begin{lem}\label{struc6} If $p=1$, then  $s=1$ and the Lie algebra $\mathfrak{g}=\mathfrak{g}_1$ is simple.
\end{lem}

\begin{proof}
Suppose the contrary, $s \geq 2$. Without loss of generality we may assume that $\mathfrak{p}_1 \subset \mathfrak{g}_1$.
Then by Lemma~\ref{struc3}, $\mathfrak{p}_2$ and $\mathfrak{p}_3$ are not subsets of $\mathfrak{g}_1$.
Hence, $\mathfrak{p}_1$ is an $\langle\boldsymbol{\cdot}\,,\boldsymbol{\cdot}\rangle$-orthogonal complement to $\varphi_1(\mathfrak{h})$ in $\mathfrak{g}_1$.
By definition of $p$, $\mathfrak{g}_i=\varphi_i(\mathfrak{h})$ for $2\leq i \leq s$. Hence, $u_i=\dim (\varphi_i(\mathbb{R}^l))=0$ for $i \geq 2$, and
$u=u_1=\dim (\varphi_1(\mathbb{R}^l))=l$.
Since $p=1$ and $u=l$ we get $\sum_{j=1}^m (v_j-1) \leq 2$ by Lemma \ref{struc1}.

If $v_j=1$ for $j=1,\dots,m$, then all $\mathfrak{g}_i=\varphi_i(\mathfrak{h})$ are ideals of  $\mathfrak{g}$ in $\mathfrak{h}$ for $2\leq i \leq s$,
that is impossible due to the effectiveness of the pair $(\mathfrak{g}, \mathfrak{h})$.

Since $v_1\geq v_2\geq v_3 \geq \cdots \geq v_m\geq 1$,
we should check the following possibilities: $(v_1,v_2)=(2,1)$, $(v_1,v_2)=(3,1)$, and $(v_1,v_2,v_3)=(2,2,1)$.
Note that all $\mathfrak{h}_j$ with $v_j=1$ are such that $\varphi_i(\mathfrak{h}_j)$ is trivial for $2\leq i \leq s$, otherwise
$\mathfrak{g}_i=\varphi_i(\mathfrak{h}_j)=\varphi_i(\mathfrak{h})$ are ideals of  $\mathfrak{g}$ in $\mathfrak{h}$.

The case $(v_1,v_2)=(2,1)$ is impossible, because $\mathfrak{p}$ in this case contains only two $\ad(\mathfrak{h})$-irreducible modules.

Let us consider the case $(v_1,v_2)=(3,1)$. If $\varphi_1(\mathfrak{h}_1)$ is trivial, then $[\mathfrak{p}_1,\mathfrak{p}_2\oplus \mathfrak{p}_3]=0$ and
$A=0$, that is impossible due to Lemma \ref{struc4}.
Hence, without loss of generality we may assume that $a^1_1=a^1_2=a^1_3=1$ and  $a^1_i=0$ for $i \geq 4$.
Then, $\mathfrak{p}_2\oplus \mathfrak{p}_3$ should coincide with an
$\langle\boldsymbol{\cdot}\,,\boldsymbol{\cdot}\rangle$-orthogonal complement to $\diag(\mathfrak{h}_1)$ in
$\varphi_1(\mathfrak{h}_1)\oplus \varphi_2(\mathfrak{h}_1)\oplus \varphi_3(\mathfrak{h}_1)\simeq 3\mathfrak{h}_1$.
But $[\mathfrak{p}_1,\mathfrak{p}_2]\subset \mathfrak{g}_1$ (since $\mathfrak{g}_1$ is an ideal in $\mathfrak{g}$), that
contradicts to $[\mathfrak{p}_1,\mathfrak{p}_2] \subset \mathfrak{p}_3$.

Finally, consider the case $(v_1,v_2,v_3)=(2,2,1)$. If $\varphi_1(\mathfrak{h}_1)$ or $\varphi_1(\mathfrak{h}_2)$ is trivial,
then $[\mathfrak{p}_1,\mathfrak{p}_2]=0$ or $[\mathfrak{p}_1,\mathfrak{p}_3]=0$ which implies
$A=0$, that is impossible due to Lemma \ref{struc4}.
Hence, without loss of generality we may assume that $a^1_1=a^1_2=1$,  $a^1_i=0$ for $i \geq 3$, $a^2_1=a^2_3=1$,  $a^3_i=0$ for other $i$.
Further,  without loss of generality, $\mathfrak{p}_2$ is an
$\langle\boldsymbol{\cdot}\,,\boldsymbol{\cdot}\rangle$-orthogonal complement to $\diag(\mathfrak{h}_1)$ in
$\varphi_1(\mathfrak{h}_1)\oplus \varphi_2(\mathfrak{h}_1)\simeq 2\mathfrak{h}_1$ and
$\mathfrak{p}_3$ is an
$\langle\boldsymbol{\cdot}\,,\boldsymbol{\cdot}\rangle$-orthogonal complement to $\diag(\mathfrak{h}_2)$ in
$\varphi_1(\mathfrak{h}_2)\oplus \varphi_3(\mathfrak{h}_2)\simeq 2\mathfrak{h}_2$.
As in the previous case, $[\mathfrak{p}_1,\mathfrak{p}_2]\subset \mathfrak{g}_1$,  that
contradicts to $[\mathfrak{p}_1,\mathfrak{p}_2] \subset \mathfrak{p}_3$. The lemma is proved.
\end{proof}

\begin{lem}\label{struc7} If $p=0$, then either $A=0$ or
$(\mathfrak{g}, \mathfrak{h})= (\mathfrak{f}\oplus \mathfrak{f}\oplus\mathfrak{f}\oplus\mathfrak{f},\diag(\mathfrak{f})=\{(X,X,X,X)\,|\, X\in f\})$
for a simple compact Lie algebra $\mathfrak{f}$.\! Moreover, up to permutation, we have
$\mathfrak{p}_1\!=\!\{(X,X,-X,-X)\,|\,X\in f\}$,
$\mathfrak{p}_2=\{(X,-X,X,-X)\,|\, X\in f\}$, and
$\mathfrak{p}_3=\{(X,-X,-X,X)\,|\, X\in f\}$.

\end{lem}

\begin{proof}
Since $p=0$, then
$u_i=\dim (\varphi_i(\mathbb{R}^l))=0$ for all $i$, and
$u=0=l$.
Since $p=0$ and $u=l$ we get $\sum_{j=1}^m (v_j-1) \leq 3$ by Lemma \ref{struc1}.
Since $\mathfrak{g}_i=\varphi_i(\mathfrak{h})$ for all $i$, then every $\mathfrak{g}_i$ is isomorphic to some simple Lie algebra
$\mathfrak{h}_j$.

If $v_j=1$ for some $j=1,\dots,m$, then all $\mathfrak{h}_j$ is an ideal of  $\mathfrak{g}$ in $\mathfrak{h}$
(indeed, there is exactly one $i\in 1,\dots,s$ with $a^j_i=1$, hence
$\mathfrak{g}_i=\varphi_i(\mathfrak{h}_j)=\varphi_i(\mathfrak{h})$),
that is impossible due to the effectiveness of the pair $(\mathfrak{g}, \mathfrak{h})$. Therefore, $v_j\geq 2$ for all $j=1,\dots,m$.
Since $\sum_{j=1}^m (v_j-1) \leq 3$, then we
we should check the following possibilities: $m=3$, $m=2$ and $m=1$.
\smallskip

If $m=3$, then $v_1=v_2=v_3=2$. It is easy to see, that (up to permutation)
$\mathfrak{p}_1$ is the
$\langle\boldsymbol{\cdot}\,,\boldsymbol{\cdot}\rangle$-orthogonal complement to $\diag(\mathfrak{h}_1)$ in
$\varphi_1(\mathfrak{h}_1)\oplus \varphi_2(\mathfrak{h}_1)\simeq 2\mathfrak{h}_1$,
$\mathfrak{p}_2$ is the
$\langle\boldsymbol{\cdot}\,,\boldsymbol{\cdot}\rangle$-orthogonal complement to $\diag(\mathfrak{h}_2)$ in
$\varphi_3(\mathfrak{h}_2)\oplus \varphi_4(\mathfrak{h}_2)\simeq 2\mathfrak{h}_2$ and
$\mathfrak{p}_3$ is the
$\langle\boldsymbol{\cdot}\,,\boldsymbol{\cdot}\rangle$-orthogonal complement to $\diag(\mathfrak{h}_3)$ in
$\varphi_5(\mathfrak{h}_3)\oplus \varphi_6(\mathfrak{h}_3)\simeq 2\mathfrak{h}_3$.
Obviously in this case we have $A=0$.
\smallskip

If $m=2$, then either $(v_1,v_2)=(2,2)$ or $(v_1,v_2)=(3,2)$.
The case $(v_1,v_2)=(2,2)$ is impossible, because $\mathfrak{p}$ in this case contains only two $\ad(\mathfrak{h})$-irreducible modules.
If $(v_1,v_2)=(3,2)$, then
$\mathfrak{p}_1\oplus \mathfrak{p}_2$ is the
$\langle\boldsymbol{\cdot}\,,\boldsymbol{\cdot}\rangle$-orthogonal complement to $\diag(\mathfrak{h}_1)$ in
$\varphi_1(\mathfrak{h}_1)\oplus \varphi_2(\mathfrak{h}_1)\oplus \varphi_3(\mathfrak{h}_1)\simeq 3\mathfrak{h}_1$,
and $\mathfrak{p}_3$ is the
$\langle\boldsymbol{\cdot}\,,\boldsymbol{\cdot}\rangle$-orthogonal complement to $\diag(\mathfrak{h}_2)$ in
$\varphi_4(\mathfrak{h}_2)\oplus \varphi_5(\mathfrak{h}_2)\simeq 2\mathfrak{h}_2$. Since $[\mathfrak{p}_1\oplus \mathfrak{p}_2,\mathfrak{p}_3]=0$,
we get $A=0$ (in fact, it is easy to prove that this variant is impossible at all).
\smallskip

If $m=1$, then we have $(\mathfrak{g}, \mathfrak{h})= (s\cdot \mathfrak{h}_1,\diag(\mathfrak{h}_1))$, and $G/H$ is a so-called Ledzer~--~Obata space, see
\cite[section 4]{LedgerObata} or \cite{Nikonorov3}. It should be noted also that for any compact Lie group $F$, a  Ledzer~--~Obata space
$F^s/\diag(F)$ is diffeomorphic to the Lie group $F^{s-1}$ \cite[P.~453]{LedgerObata}.

It is known, that the module $\mathfrak{p}$ decomposed in this case into the sum of $s-1$ pairwise $\ad(\mathfrak{h})$-isomorphic irreducible summand,
but such a decomposition is not unique, see \cite{Nikonorov3}.
Hence, $s=4$ and $(\mathfrak{g}, \mathfrak{h})= (\mathfrak{f}\oplus \mathfrak{f}\oplus\mathfrak{f}\oplus\mathfrak{f},\diag(\mathfrak{f}))$
for some simple Lie algebra $\mathfrak{f}(=\mathfrak{h}_1)$.

Clear that $\mathfrak{h}=\{(X,X,X,X)\,|\, X\in f\}$. Any $\ad(\mathfrak{h})$-irreducible module in $\mathfrak{p}$ has the form
$\mathfrak{q}=\{(\alpha_1X,\alpha_2X,\alpha_3X,\alpha_4X)\,|\, X\in f\}$, where $\alpha=(\alpha_1,\alpha_2,\alpha_3,\alpha_4)\in \mathbb{R}^4$
is of unit length and satisfies $\alpha_1+\alpha_2+\alpha_3+\alpha_4=0$.
Therefore, $\mathfrak{p}$ could be decomposed into a sum of $\ad(\mathfrak{h})$-irreducible and pairwise
$\langle\boldsymbol{\cdot}\,,\boldsymbol{\cdot}\rangle$-orthogonal modules only as follows (see details in \cite{Nikonorov3}):
$\mathfrak{p}=\mathfrak{p}_1\oplus \mathfrak{p}_2\oplus \mathfrak{p}_3$, where
\begin{eqnarray*}
\mathfrak{p}_1=\{(\alpha_1X,\alpha_2X,\alpha_3X,\alpha_4X)\,|\, X\in f\},\quad\quad
\mathfrak{p}_2=\{(\beta_1X,\beta_2X,\beta_3X,\beta_4X)\,|\, X\in f\},\\
\mathfrak{p}_3=\{(\gamma_1X,\,\gamma_2X,\,\gamma_3X,\,\gamma_4X)\,|\, X\in f\},\quad\quad\quad\quad\quad\quad\,\,
(\alpha,\alpha)=(\beta,\beta)=(\gamma,\gamma)=1,\\
\alpha_1+\alpha_2+\alpha_3+\alpha_4=0,\quad\quad\,\,\,\,\, \beta_1+\beta_2+\beta_3+\beta_4=0,\quad\quad\,\,\,\,\, \gamma_1+\gamma_2+\gamma_3+\gamma_4=0,\\
(\alpha,\beta)=(\alpha,\gamma)=(\beta,\gamma)=0,\,\, \mbox{where} \,\,\,(x,y)=x_1y_1+x_2y_2+x_3y_3+x_4y_4, \,\, x,y \in \mathbb{R}^4.
\end{eqnarray*}
Since $[\mathfrak{p}_i,\mathfrak{p}_i]\subset \mathfrak{h}$, $i=1,2,3$, then
$|\alpha_1|=|\alpha_2|=|\alpha_3|=|\alpha_4|=|\beta_1|=|\beta_2|=|\beta_3|=|\beta_4|=|\gamma_1|=|\gamma_2|=|\gamma_3|=|\gamma_4|=1/2$.
Therefore, up to permutation, we have $\alpha=(1/2,1/2,-1/2,-1/2)$, $\beta=(1/2,-1/2,1/2,-1/2)$, $\gamma=(1/2,-1/2,-1/2,1/2)$.
The lemma is proved.
\end{proof}
\smallskip

From the previous results of this section we immediately get

\begin{theorem}\label{struc8}
Let $G/H$ be a generalized Wallach space with connected $H$. Then one of the following assertions holds:

{\rm 1)} $G/H$ is locally a direct product of three irreducible symmetric spaces of compact type;

{\rm 2)} The group $G$ is simple;

{\rm 3)} On the Lie algebra level,
$(\mathfrak{g}, \mathfrak{h})= (\mathfrak{f}\oplus \mathfrak{f}\oplus\mathfrak{f}\oplus\mathfrak{f},\diag(\mathfrak{f})=\{(X,X,X,X)\,|\, X\in f\})$
for a simple compact Lie algebra $\mathfrak{f}$ and, up to permutation, we have
$\mathfrak{p}_1=\{(X,X,-X,-X)\,|\, X\in f\}$,
$\mathfrak{p}_2=\{(X,-X,X,-X)\,|\, X\in f\}$, and
$\mathfrak{p}_3=\{(X,-X,-X,X)\,|\, X\in f\}$.
\end{theorem}

\begin{pred}\label{struc9} Let $G/H$ be a generalized Wallach space such that
$(\mathfrak{g}, \mathfrak{h})= (\mathfrak{f}\oplus \mathfrak{f}\oplus\mathfrak{f}\oplus\mathfrak{f},\diag(\mathfrak{f})=\{(X,X,X,X)\,|\, X\in f\})$
for a simple compact Lie algebra $\mathfrak{f}$ and
$\mathfrak{p}_1=\{(X,X,-X,-X)\,|\, X\in f\}$,
$\mathfrak{p}_2=\{(X,-X,X,-X)\,|\, X\in f\}$, and
$\mathfrak{p}_3=\{(X,-X,-X,X)\,|\, X\in f\}$.
Then $A=\frac{1}{4}\dim(\mathfrak{f})$ and $a_1=a_2=a_3=\frac{1}{4}$.
\end{pred}

\begin{proof}
Let $e_i$, $i=1,\dots,\dim(\mathfrak{f})$, be an orthonormal with respect to $-B_{\mathfrak{f}}$ (the minus Killing form of the Lie algebra $\mathfrak{f}$).
Then $1/2(e_i,e_i,-e_i,-e_i)$, $1/2(e_i,-e_i,e_i,-e_i)$, and $1/2(e_i,-e_i,-e_i,e_i)$, $i=1,\dots,\dim(\mathfrak{f})$,
forms $\langle\boldsymbol{\cdot}\,,\boldsymbol{\cdot}\rangle$-orthonormal bases in $\mathfrak{p}_1$, $\mathfrak{p}_1$, and $\mathfrak{p}_3$ respectively.
Therefore,
\begin{eqnarray*}
A=\sum\limits_{i,j,k=1}^{\dim(\mathfrak{f})}
\left\langle \left[\frac{1}{2}(e_i,e_i,-e_i,-e_i), \frac{1}{2}(e_j,-e_j,e_j,-e_j)\right], \frac{1}{2}(e_k,-e_k,-e_k,e_k)\right\rangle^2= \\
\frac{1}{64} \sum\limits_{i,j,k=1}^{\dim(\mathfrak{f})}
\left\langle \left[(e_i,e_i,-e_i,-e_i), (e_j,-e_j,e_j,-e_j)\right],(e_k,-e_k,-e_k,e_k)\right\rangle^2= \\
\frac{1}{64}\sum\limits_{i,j,k=1}^{\dim(\mathfrak{f})} 16\cdot (-B_{\mathfrak{f}}([e_i,e_j],e_k))^2=
\frac{1}{4}\sum\limits_{i,j=1}^{\dim(\mathfrak{f})} (-B_{\mathfrak{f}}([e_i,e_j],[e_i,e_j]))=\\
-\frac{1}{4}\sum\limits_{i,j=1}^{\dim(\mathfrak{f})} (-B_{\mathfrak{f}}([e_i,[e_i,e_j]],e_j]))=
-\frac{1}{4}\sum\limits_{i=1}^{\dim(\mathfrak{f})} \trace(\ad(e_i) \cdot \ad(e_i))=\\
-\frac{1}{4}\sum\limits_{i=1}^{\dim(\mathfrak{f})} B_{\mathfrak{f}}(e_i,e_i)=\frac{1}{4}\dim(\mathfrak{f}).
\end{eqnarray*}
Here we have used the definition of the Killing form: $B_{\mathfrak{f}}(X,Y)=\trace (\ad(X) \cdot \ad(Y))$ and the fact that all operators $\ad(X)$ are
skew-symmetric with respect to $B_{\mathfrak{f}}$.

Since $\dim(\mathfrak{p}_1)=\dim(\mathfrak{p}_2)=\dim(\mathfrak{p}_3)=\dim(\mathfrak{f})$, then $a_1=a_2=a_3=1/4$.
Note also, that this result follows also from more general calculations for an arbitrary Ledger~--~Obata space in \S~4 of~\cite{Nikonorov3}.
\end{proof}

{\small
\renewcommand{\arraystretch}{1.3}
\begin{table}[p]
{\bf Table 2.} The pairs $(\mathfrak{g},\mathfrak{h})$ corresponded to $\mathbb{Z}_2\times\mathbb{Z}_2$-groups in $\Aut(\mathfrak{g})$
with simple compact $\mathfrak{g}$.
\begin{center}
\begin{tabular}{|c|c|c|c|c|c|}
\hline
N&  $\mathfrak{g}$& $\mathfrak{h}$ & $\mathfrak{k}_1$ &  $\mathfrak{k}_2$ & $\mathfrak{k}_3$ \\
\hline\hline
1& $su(p+q)$ & $so(p)\oplus so(q)$ & $so(p+q)$ & $so(p+q)$ & $s(u(p)\oplus u(p))$ \\  \hline
2& $su(2p)$ & $u(p)$ & $so(2p)$ & $sp(p)$ & $s(u(p)\oplus u(p))$ \\  \hline
3& $su(2p+2q)$ & $sp(p)\oplus sp(q)$ & $sp(p+q)$ & $sp(p+q)$ & $s(u(2p)\oplus u(2q))$ \\  \hline
4& $\!\!su(p\!+\!q\!+\!r\!+\!s)\!\!\!$ &\!\!$s(u(p)\!\oplus\!u(q)\!\oplus\!u(r)\!\oplus\!u(s))$\!\!\!& \!\!$s(u(p\!+\!q)\!\oplus \!u(r\!+\!s))$\!\!&
\!\!$s(u(p\!+\!r)\!\oplus\!u(q\!+\!s))$\!\!&  \!\!$s(u(p\!+\!s)\!\oplus\!u(q\!+\!r))$\!\!\\  \hline
5& $su(2p)$ & $su(p)$ & $s(u(p)\oplus u(p))$ & $s(u(p)\oplus u(p))$ & $s(u(p)\oplus u(p))$ \\  \hline
6& $\!\!so(p\!+\!q\!+\!r\!+\!s)\!\!$ & $\!\!so(p)\!\oplus\!so(q)\!\oplus\!so(r)\!\oplus \!so(s)\!\!$ & \!\!$so(p\!+\!q)\!\oplus \!so(r\!+\!s)$\!\! &
\!\!$so(p\!+\!r)\!\oplus\!so(q\!+\!s)$\!\! &  \!\!$so(p\!+\!s)\!\oplus\!so(q\!+\!r)$\!\!\\  \hline
7& $so(2p)$ & $so(p)$ & $so(p)\oplus so(p)$ & $so(p)\oplus so(p)$ & $u(p)$ \\  \hline
8& $so(2p+2q)$ & $u(p)\oplus u(q)$ & $so(2p)\oplus so(2q)$ & $u(p+q)$ & $u(p+q)$ \\  \hline
9& $so(4p)$ & $sp(p)$ & $u(2p)$ & $u(2p)$ & $u(2p)$ \\  \hline
10& $sp(p)$ & $so(p)$ & $u(p)$ & $u(p)$ & $u(p)$ \\  \hline
11& $sp(p+q)$ & $u(p)\oplus u(q)$ & $u(p+q)$ & $u(p+q)$ & $sp(p)\oplus sp(q)$ \\  \hline
12& $sp(2p)$ & $sp(p)$ & $u(2p)$ & $sp(p)\oplus sp(p)$ & $sp(p)\oplus sp(p)$ \\  \hline
13& $\!\!sp(p\!+\!q\!+\!r\!+\!s)\!\!$ & $\!\!sp(p)\!\oplus\!sp(q)\!\oplus\!sp(r)\!\oplus \!sp(s)\!\!$ & \!\!$sp(p\!+\!q)\!\oplus \!sp(r\!+\!s)$\!\! &
\!\!$sp(p\!+\!r)\!\oplus\!sp(q\!+\!s)$\!\! &  \!\!$sp(p\!+\!s)\!\oplus\!sp(q\!+\!r)$\!\!\\  \hline
14& $e_6$ & $2su(3)\oplus \mathbb{R}^2$ & $su(6)\oplus sp(1)$ & $su(6)\oplus sp(1)$ & $su(6)\oplus sp(1)$ \\  \hline
15& $e_6$ & $su(4)\oplus2sp(1)\oplus \mathbb{R}$ & $su(6)\oplus sp(1)$ & $su(6)\oplus sp(1)$ & $so(10)\oplus \mathbb{R}$ \\  \hline
16& $e_6$ & $su(5)\oplus \mathbb{R}^2$ & $su(6)\oplus sp(1)$ & $so(10)\oplus \mathbb{R}$ & $so(10)\oplus \mathbb{R}$ \\  \hline
17& $e_6$ & $so(8)\oplus \mathbb{R}^2$ & $so(10)\oplus \mathbb{R}$ & $so(10)\oplus \mathbb{R}$ & $so(10)\oplus \mathbb{R}$ \\  \hline
18& $e_6$ & $sp(3)\oplus sp(1)$ & $su(6)\oplus sp(1)$ & $f_4$ & $sp(4)$ \\  \hline
19& $e_6$ & $so(6)\oplus \mathbb{R}$ & $su(6)\oplus sp(1)$ & $sp(4)$ & $sp(4)$ \\  \hline
20& $e_6$ & $so(9)$ & $so(10)\oplus \mathbb{R}$ & $f_4$ & $f_4$ \\  \hline
21& $e_6$ & $so(5)\oplus so(5)$ & $so(10)\oplus \mathbb{R}$ & $sp(4)$ & $sp(4)$ \\  \hline
22& $e_7$ & $su(6)\oplus \mathbb{R}^2$ & $so(12)\oplus sp(1)$ & $so(12)\oplus sp(1)$ & $so(12)\oplus sp(1)$ \\  \hline
23& $e_7$ & $so(8)\oplus 3sp(1)$ & $so(12)\oplus sp(1)$ & $so(12)\oplus sp(1)$ & $so(12)\oplus sp(1)$ \\  \hline
24& $e_7$ & $so(10)\oplus \mathbb{R}^2$ & $so(12)\oplus sp(1)$ & $e_6\oplus \mathbb{R}$ & $e_6\oplus \mathbb{R}$ \\  \hline
25& $e_7$ & $su(6)\oplus sp(1) \oplus \mathbb{R}$ & $so(12)\oplus sp(1)$ & $e_6\oplus \mathbb{R}$ & $su(8)$ \\  \hline
26& $e_7$ & $su(4)\oplus su(4) \oplus \mathbb{R}$ & $so(12)\oplus sp(1)$ & $su(8)$ & $su(8)$ \\  \hline
27& $e_7$ & $f_4$ & $e_6\oplus \mathbb{R}$ & $e_6\oplus \mathbb{R}$ & $e_6\oplus \mathbb{R}$ \\  \hline
28& $e_7$ & $sp(4)$ & $e_6\oplus \mathbb{R}$ & $su(8)$ & $su(8)$ \\  \hline
29& $e_7$ & $so(8)$ & $su(8)$ & $su(8)$ & $su(8)$ \\  \hline
30& $e_8$ & $e_6 \oplus \mathbb{R}^2$ & $e_7 \oplus sp(1)$ & $e_7 \oplus sp(1)$ & $e_7 \oplus sp(1)$ \\  \hline
31& $e_8$ & $so(12)\oplus 2sp(1)$ & $e_7 \oplus sp(1)$ & $e_7 \oplus sp(1)$ & $so(16)$ \\  \hline
32& $e_8$ & $su(8) \oplus \mathbb{R}$ & $e_7 \oplus sp(1)$ & $so(16)$ & $so(16)$ \\  \hline
33& $e_8$ & $so(8)\oplus so(8)$ & $so(16)$ & $so(16)$ & $so(16)$ \\  \hline
34& $f_4$ & $su(3)\oplus \mathbb{R}^2$ & $sp(3)\oplus sp(1)$ & $sp(3)\oplus sp(1)$ & $sp(3)\oplus sp(1)$ \\  \hline
35& $f_4$ & $so(5)\oplus 2sp(1)$ & $sp(3)\oplus sp(1)$ & $sp(3)\oplus sp(1)$ & $so(9)$ \\  \hline
36& $f_4$ & $so(8)$ & $so(9)$ & $so(9)$ & $so(9)$ \\  \hline
37& $g_2$ & $\mathbb{R}^2$ & $sp(1)\oplus sp(1)$ & $sp(1)\oplus sp(1)$ & $sp(1)\oplus sp(1)$ \\  \hline

\end{tabular}
\end{center}
\end{table}
\renewcommand{\arraystretch}{1}
}

\section{Generalized Wallach spaces and $\mathbb{Z}_2\times\mathbb{Z}_2$-symmetric spaces}\label{simple}

Let $\Gamma$ be a finite abelian subgroup of the automorphism group of a Lie group $G$.
Then the homogeneous space $G/H$ is called a $\Gamma$-symmetric space, if $(G^{\Gamma})_0 \subset H \subset G^{\Gamma}$,
where the subgroup $G^{\Gamma}$ consists of elements of $G$ invariant with respect to $\Gamma$, and
$(G^{\Gamma})_0$ is its unit component \cite{Lutz}. We get symmetric spaces for $\Gamma =\mathbb{Z}_2$ and
$k$\,-symmetric spaces for $\mathbb{Z}_k$ \cite{WolfGray}. For the Klein four-group $\mathbb{Z}_2\times\mathbb{Z}_2$, the above definition
give us $\mathbb{Z}_2\times\mathbb{Z}_2$-symmetric spaces, which were studied in \cite{Bahturin} and \cite{Kollross},
In particular, a classification of these spaces for simple compact groups $G$ were obtained in these two papers.

Another approach for this classification was applied in the paper
\cite{Huang}.
On the Lie algebra level this classification is equivalent to the classification of $\mathbb{Z}_2\times\mathbb{Z}_2$-groups in the automorphism group
$\Aut(\mathfrak{g})$ for all simple compact Lie algebra $\mathfrak{g}$. It is clear that any $\mathbb{Z}_2\times\mathbb{Z}_2$-group
are generated with two commuting involutive automorphisms of $\mathfrak{g}$.
\smallskip

By Theorem \ref{theorem1}, every generalized Wallach spaces $G/H$
with simple $G$ is a $\mathbb{Z}_2\times\mathbb{Z}_2$-symmetric space.
Hence, we obtain {\bf the following algorithm}.
We should consider a complete list of $\mathbb{Z}_2\times\mathbb{Z}_2$-symmetric spaces. It is useful to deal with
a such classification on the Lie algebra level, i.~e. with the classification of simple compact Lie algebras $\mathfrak{g}$ with
$\mathbb{Z}_2\times\mathbb{Z}_2$-groups $\Gamma$ in $\Aut(\mathfrak{g})$. A list of such objects is given e.~g. in \cite{Huang} (see Tables 3 and 4 there).
We reproduce it in our Table 2. We denote by $\mathfrak{h}$ a Lie subalgebra of $\mathfrak{g}$ consisted of fixed points of the corresponding
group $\Gamma$. By $\mathfrak{k}_1$, $\mathfrak{k}_2$, and $\mathfrak{k}_3$ we denote symmetric subalgebras in $\mathfrak{g}$,
that consist respectively of fixed points of involutions $\sigma_1$, $\sigma_2$, and $\sigma_3=\sigma_1\sigma_2$, such that
$\sigma_1$ and $\sigma_2$ generate $\Gamma$. This information could be easily derived from Tables 2, 3, and 4 of the paper \cite{Huang}.
Of course, we consider these subalgebras up to permutation. Recall also that all the pairs $(\mathfrak{k}_i, \mathfrak{h})$, $i=1,2,3$, are also symmetric.

Our final step is the following. For each line of Table 2, we should determine the ``effective parts`` $(\tilde{\mathfrak{k}}_i, \tilde{\mathfrak{h}}_i)$
of the pairs $(\mathfrak{k}_i, \mathfrak{h})$, $i=1,2,3$. This means that we need to eliminate nontrivial ideals of $\mathfrak{g}$ from
$\mathfrak{h}$. More precisely, let $\mathfrak{a}$ be a maximal ideal of $\mathfrak{k}_i$ in $\mathfrak{h}$,
then $\tilde{\mathfrak{k}}_i$ (respectively, $\tilde{\mathfrak{h}}_i$) is \linebreak
a~$B_{\mathfrak{g}}$-orthogonal compliment to $\mathfrak{a}$
in $\mathfrak{k}_i$ (respectively, $\mathfrak{h}$), where $B_{\mathfrak{g}}$ is the Killing form of the Lie algebra $\mathfrak{g}$.
Further, we should check that $\mathfrak{p}_i$, a $B_{\mathfrak{g}}$-orthogonal compliment to $\tilde{\mathfrak{h}}_i$
in $\tilde{\mathfrak{k}}_i$ (or, equivalently, a $B_{\mathfrak{g}}$-orthogonal compliment to ${\mathfrak{h}}$
in ${\mathfrak{k}}$) is $\ad(\tilde{\mathfrak{h}}_i)$-irreducible (or, equivalently, $\ad(\mathfrak{h})$-irreducible).
We have the following obvious result

\begin{lem}\label{irreducible} Let $(\mathfrak{g},\mathfrak{h})$ be a pair from Table 2. Then the following conditions are equivalent:

{\rm 1)} $(\mathfrak{g},\mathfrak{h})$ corresponds to a generalized Wallach space $G/H$ with connected $H$;

{\rm 2)} the modules $\mathfrak{p}_i$, $i=1,2,3$, are $\ad(\mathfrak{h})$-irreducible;

{\rm 3)} the symmetric pairs $(\tilde{\mathfrak{k}}_i,\tilde{\mathfrak{h}}_i)$, $i=1,2,3$, are irreducible.
\end{lem}

Removing from Table 2 all pairs that do not satisfy the conditions of Lemma \ref{irreducible}
we get the Table~1, that contains all
possible  $(\mathfrak{g}, \mathfrak{h})$ corresponded to generalized Wallach space with simple groups $G$.

\begin{theorem}\label{simplegroup} For any generalized Wallach space $G/H$ with simple $G$ and connected $H$, the pair $(\mathfrak{g},\mathfrak{h})$
is in Table 2. A pair $(\mathfrak{g},\mathfrak{h})$ from Table 2 generates a generalized Wallach space
if and only if it is in Table 1.
\end{theorem}

\begin{proof}
The first assertion we get immediately from Theorem \ref{theorem1}.
Let us prove the second assertion.
We have to check all pairs from Table 2, using Lemma \ref{irreducible}. For a list of irreducible symmetric pairs see e.~g. \cite{Helgason1978} or \cite{Wolf2011}.

Let us consider {\bf line 1}.
It is easy to see that the pairs
$(\tilde{\mathfrak{k}}_1, \tilde{\mathfrak{h}}_1)$ and
$(\tilde{\mathfrak{k}}_2, \tilde{\mathfrak{h}}_2)$ are irreducible symmetric, but the pair
$(\tilde{\mathfrak{k}}_3, \tilde{\mathfrak{h}}_3)=(s(u(p)\oplus u(p)),so(r)\oplus so(q))$ is not. Hence the pair $(\mathfrak{g}, \mathfrak{h})$ does not generate
a generalized  Wallach space in this case.

Let us consider {\bf line 2}.
The pairs
$(\tilde{\mathfrak{k}}_1, \tilde{\mathfrak{h}}_1)$,
$(\tilde{\mathfrak{k}}_2, \tilde{\mathfrak{h}}_2)$, and
$(\tilde{\mathfrak{k}}_3, \tilde{\mathfrak{h}}_3)=\bigl(su(p)\oplus su(p),\diag(su(p))\bigr)$ are irreducible symmetric and we
get  line 4 in Table~1.

For {\bf line 3}, the pairs
$(\tilde{\mathfrak{k}}_1, \tilde{\mathfrak{h}}_1)$ and
$(\tilde{\mathfrak{k}}_2, \tilde{\mathfrak{h}}_2)$ are irreducible symmetric, but the pair
$(\tilde{\mathfrak{k}}_3, \tilde{\mathfrak{h}}_3)=(s(u(2p)\oplus u(2q)),sp(p)\oplus sp(q))$ is not.

For {\bf line 5}, all the pairs
$(\tilde{\mathfrak{k}}_i, \tilde{\mathfrak{h}}_i)$, $i=1,2,3$, are not
irreducible. The same is true for the {\bf lines 9} and {\bf 10}.

Let us check {\bf line 6} in Table 2.
In this case we have
$$
(\mathfrak{g}, \mathfrak{h})=(so(p+q+r+s),so(p)\oplus so(q)\oplus so(r)\oplus so(s)),
$$
$$\mathfrak{k}_1=so(p+q)\oplus so(r+s),\quad \mathfrak{k}_2=so(p+r)\oplus so(q+s), \quad \mathfrak{k}_3=so(p+s)\oplus so(q+r).
$$
It is easy to see that the symmetric pair $(\mathfrak{k}_i,\mathfrak{h})$ is
decomposed into the sum of two symmetric pairs provided that $p\cdot q\cdot r\cdot s \neq 0$. In order to get $\ad(\mathfrak{h})$-irreducible modules
$\mathfrak{p}_i$, we should put  $s=0$ (without loss of generality). Then we get
$(\tilde{\mathfrak{k}}_1, \tilde{\mathfrak{h}}_1)=(so(p+q),so(p)\oplus so(q))$,
$(\tilde{\mathfrak{k}}_2, \tilde{\mathfrak{h}}_2)=(so(p+r),so(p)\oplus so(r))$, and
$(\tilde{\mathfrak{k}}_3, \tilde{\mathfrak{h}}_3)=(so(r+q),so(r)\oplus so(q))$.
We see that the modules $\mathfrak{p}_i$, $i=1,2,3$, are $\ad(\mathfrak{h})$-irreducible for all $p,q,r \geq 1$ and $s=0$.
Hence we get line 1 of Table 1. Note that for $p=q=r=1$ we get the Lie algebra $so(3)\simeq su(2)\simeq sp(1)$ that generate the group
$SU(2)$, a $3$-dimensional generalized Wallach space.

Applying the same argument to {\bf lines 4} and {\bf 13} in Table 2 we get lines 2 and 3 in Table~1.

For {\bf line 7}, the pair
$(\tilde{\mathfrak{k}}_3, \tilde{\mathfrak{h}}_3)=(u(p),so(p))$ is not
irreducible.

For {\bf line 8}, the pairs
$(\tilde{\mathfrak{k}}_2, \tilde{\mathfrak{h}}_2)$ and
$(\tilde{\mathfrak{k}}_3, \tilde{\mathfrak{h}}_3)$ (that coincide with $\bigl(su(p+q),s(u(p)\oplus u(q))\bigr)$) are irreducible symmetric,
the pair $({\mathfrak{k}}_1, {\mathfrak{h}})=(so(2p)\oplus so(2q),u(p)\oplus u(q))$ is irreducible only if $q=1$ or $p=1$.
Hence, we get the line 5 in Table~1.

For {\bf line 11}, the pair
$(\tilde{\mathfrak{k}}_3, \tilde{\mathfrak{h}}_3)=(sp(p)\oplus sp(q),u(p)\oplus u(q))$ is not
irreducible.

For {\bf line 12}, the pair
$(\tilde{\mathfrak{k}}_1, \tilde{\mathfrak{h}}_1)=(u(2p),sp(p))$ is not
irreducible.
\smallskip

By the same manner we check {\bf lines 13-37}, corresponded to exceptional Lie algebras $\mathfrak{g}$.
We do not write all details here, because this is a direct and easy procedure.
Recall the following isomorphisms between Lie algebras, which simplify the mentioned check:
$sp(1) \simeq su(2) \simeq so(3)$, $sp(2) \simeq so(5)$, and $su(4)\simeq so(6)$.

Note that the pairs in {\bf lines 15, 17, 18, 23, 25, 29, 31, 33, 35}, and {\bf 36} generate generalized  Wallach spaces (see lines 6--15 of Table~1).
All other pairs $(\mathfrak{g},\mathfrak{h})$ with exceptional $\mathfrak{g}$ from Table 2 are such that at least one of the pairs
$(\tilde{\mathfrak{k}}_i, \tilde{\mathfrak{h}}_i)$, $i=1,2,3$, is not irreducible symmetric.
As a final result, we completes Table 1.
\end{proof}

\section{Calculation of $a_1$, $a_2$, $a_3$}\label{calculation}

Let $G/H$ be a compact homogeneous space with semisimple $G$, then the Killing form $B$
of the Lie algebra $\mathfrak{g}$ is negatively defined, and $\mathfrak{p}$, the $B$-orthogonal complement to $\mathfrak{h}$ in $\mathfrak{g}$,
could be naturally identified with the tangent space of $G/H$ at the point $eH$.
Note, that for any Riemannian invariant metric $\rho$ on  $G/H$,
the isotropy representation $\tau: H \mapsto \End(\mathfrak{p})$ of the isotropy group $H$  is such that every $\tau(a)=\Ad(a)|_{\mathfrak{p}}$
is orthogonal transformation.
Moreover, the isotropy representation $d\tau:\mathfrak{h} \mapsto \End(\mathfrak{p})$ of the isotropy algebra $\mathfrak{h}$
is such that every $d\tau(f)=\ad(f)|_{\mathfrak{p}}$ is a skew-symmetric.

For the inner product $\langle\boldsymbol{\cdot}\,,\boldsymbol{\cdot}\rangle|_{\mathfrak{h}}=-B|_{\mathfrak{h}}$ on $\mathfrak{h}$ we can consider
the~Casimir operator $C$ of the~(restruction of the) adjoint representation of~$\mathfrak{h}$ on~$\mathfrak{p}$.
Let~$\big\{e^j_0\big\}$ be an~orthonormal basis in~$\mathfrak{h}$ with
respect to $\langle\boldsymbol{\cdot}\,,\boldsymbol{\cdot}\rangle$, $1\le j\le\dim(\mathfrak{h})$, then (see e.~g. \cite[7.88]{Bes})
$$
C=-\sum\limits_{1\le j\le\dim(\mathfrak{h})} \ad \big(e^j_0\big)\big|_{ \mathfrak{p}} \circ \ad \big(e^j_0\big)\big|_{ \mathfrak{p}}.
$$
For any $\ad(\mathfrak{h})$-irreducible submodule $\mathfrak{q} \subset \mathfrak{p}$, the operator $C$ is proportional to the identity operator.
If $\mathfrak{p}$ is the sum of $\ad(\mathfrak{h})$-irreducible submodules $\mathfrak{p}_i$,
then we have $C|_{ \mathfrak{p}_i}=c_i \Id|_{ \mathfrak{p}_i}$ for some constant~$c_i$, that are called the Casimir constants.
Since $\ad \big(e^j_0\big)\big|_{ \mathfrak{p}}$ is skew-symmetric, then
\begin{equation}\label{casim}
c_i=\sum_{1\le j\le\dim(\mathfrak{h})} \left\langle [e^j_0,e],[e^j_0,e] \right\rangle
\end{equation}
for an~arbitrary unit (with respect to $\langle\boldsymbol{\cdot}\,,\boldsymbol{\cdot}\rangle=-B$) vector $e$ in~$\mathfrak{p}_i$. In particular,
$c_i \geq 0$.
\medskip

Now, we continue to study generalized Wallach spaces.
Recall one important property of the numbers $[ijk]$,
see (\ref{constwz}).
According to lemma 1.5 in~\cite{WZ3}, we get the~formula
$$
\sum_{j,k}[ijk]=d_i(1-2c_i),
$$
for all $i=1,2,3$, where $c_i$ is the~corresponding Casimir constant, $d_i=\dim({\mathfrak{p}}_i)$.
Using the above consideration we can rewrite this equality as follows (see (\ref{consta})):
\begin{equation}\label{sumconst}
2A=[ijk]+[ikj]=d_i(1-2c_i), \quad i\neq j \neq k \neq i.
\end{equation}
Hence we get the following result (obtained in \cite{Nikonorov2} and \cite{Lomshakov2}).

\begin{lem}\label{calca}
For a generalized Wallach space, we have
$d_i\ge 2A$ for all $i=1,2,3$. Moreover, the~equality $d_i=2A$ is equivalent to the~condition
$[\mathfrak{h},\mathfrak{p}_i]=0$.
\end{lem}
\bigskip

Now, we~give a~convenient method for calculating~$c_i$ and~$A$ for generalized Wallach spaces $G/H$ with simple $G$.
Consider a~connected Lie subgroup~$K_i$ in~$G$
with Lie algebra~$\mathfrak{k}_i=\mathfrak{h}\oplus \mathfrak{p}_i$ as in (\ref{subalgi}). It is clear, that the~homogeneous spaces
~$K_i/H$ and~$G/K_i$ are locally symmetric (see Section~\ref{waau} and~\cite[7.70]{Bes}). If~$K_i$
does not act almost effectively on~$M=K_i/H$, consider its subgroup $\widetilde{\!K}_i$ acting
on $M=K_i/H={}\,\widetilde{\!K}_i/{}\,\widetilde{\!H}_i$ almost effectively
(here we denote by ${}\,\widetilde{\!H_i}$ the~corresponding isotropy group).
The~pair of the corresponding Lie algebras $\big(\tilde{\mathfrak{k}}_i,\tilde{\mathfrak{h}}_i\big)$ is irreducible symmetric
(see~\cite[7.100]{Bes}).
A more direct and convenient way to produce the pair $\big(\tilde{\mathfrak{k}}_i,\tilde{\mathfrak{h}}_i\big)$ is the following:
If $\mathfrak{a}$ is a maximal ideal of $\mathfrak{k}_i$ in $\mathfrak{h}$,
then $\tilde{\mathfrak{k}}_i$ (respectively, $\tilde{\mathfrak{h}}_i$) is a $\langle\boldsymbol{\cdot}\,,\boldsymbol{\cdot}\rangle$-orthogonal
compliment to $\mathfrak{a}$
in $\mathfrak{k}_i$ (respectively, $\mathfrak{h}$).
\bigskip

If~$\tilde{\mathfrak{k}}_i$ is a~simple Lie algebra then its Killing
form~$B_{\tilde{\mathfrak{k}}_i}$ is proportional to the~restriction of the~Killing form~$B$
of~$\mathfrak{g}$ to~$\tilde{\mathfrak{k}}_i$. Therefore, there exists a~positive number~${\gamma}_i$
with the~property
\begin{equation}\label{dynkinind1}
B_{\tilde{\mathfrak{k}}_i}={\gamma}_i \cdot B\bigr\vert_{\tilde{\mathfrak{k}}_i}.
\end{equation}

\begin{lem}\label{constcalc}
In the notations as above, we have $c_i=\gamma_i/2$ and $A=d_i(1-{\gamma}_i)/2$.
\end{lem}

\begin{proof} Clearly, $\tilde{\mathfrak{k}}_i=\tilde{\mathfrak{h}}_i\oplus \mathfrak{p}_i$. Since for a locally
symmetric space, the~Casimir constant is equal to~$1/2$ (see~\cite[7.93]{Bes}),
we may calculate~$c_i$ as follows. Consider any~$\langle\boldsymbol{\cdot}\,,\boldsymbol{\cdot}\rangle$-orthonormal basis
$\big\{e_j^0\big\}$ in~$\mathfrak{h}$, such
that~$e_j^0\in\tilde{\mathfrak{h}}_i$ for $0\le j\le\dim(\tilde{\mathfrak{h}})$ and $\langle e_j^0, \tilde{\mathfrak{h}}_i \rangle=0$
for $j>\dim(\tilde{\mathfrak{h}})$. Obviously,
$\big[e_j^0,e\big]=0$ for all $e\in \mathfrak{p}_i$ and
$j>\dim(\tilde{\mathfrak{h}})$. Therefore,
$$
c_i=\sum_{0\le j\le\dim(\mathfrak{h})} \left\langle [e_j^0,e],[e_j^0,e] \right\rangle=
\sum_{0\le j\le \dim(\tilde{\mathfrak{h}})} \left\langle [e_j^0,e],[e_j^0,e] \right\rangle
$$
for every unit vector $e\in \mathfrak{p}_i$. Consider
the~vectors $f_j^{\,0}=\frac1{\sqrt{{\gamma}_i}}e^0_j$, $1\le i\le \dim(\tilde{\mathfrak{h}})$. They form an~orthonormal basis in~$\tilde{\mathfrak{h}}_i$
with respect to~$-B_{\tilde{\mathfrak{k}}}$.
Suppose that $\tilde{e}=\frac1{\sqrt{{\gamma}_i}}e$, where $e$ is a vector of unit length
with respect to
$-B(\boldsymbol{\cdot}\,,\boldsymbol{\cdot})=
\langle\boldsymbol{\cdot}\,,\boldsymbol{\cdot}\rangle$.
Then, the~Casimir constant
($=1/2$) of the~adjoint representation of~$\tilde{\mathfrak{h}}_i$ on~$\mathfrak{p}_i$ satisfies the following equality:
\begin{eqnarray*}
\frac{1}{2}=\sum_{0\le j\le \dim(\tilde{\mathfrak{h}})}
-B_{\mathfrak{k}_i}([f_j^{\,0},\tilde{e}\,],[f_j^{\,0},\tilde{e}\,])=\\
\sum_{0\le j\le \dim(\tilde{\mathfrak{h}})}
\gamma_i \left\langle [f_j^{\,0},\tilde{e}\,],[f_j^{\,0},\tilde{e}\,]\right\rangle=\frac{1}{\gamma_i}
\sum_{0\le j\le \dim(\tilde{\mathfrak{h}})} \left\langle [e_j^0,e],[e_j^0,e] \right\rangle=\frac{c_i}{{\gamma}_i}.
\end{eqnarray*}
Therefore, ${\gamma}_i=2 {c_i}$.
Furthermore, since $2A=d_i(1-2c_i)$, we have $A=d_i(1-{\gamma}_i)/2$.
The~lemma is proved.
\end{proof}

\begin{remark}
Since $a_i=A/d_i$, $i=1,2,3$, then $a_1=a_2=a_3$ if and only if $c_1=c_2=c_3$. Note that the last equality holds if and only if the Killing (the standard)
metric on the space $G/H$ is Einstein \cite[7.92]{Bes}. Therefore, the equality $a_1=a_2=a_3$ means the same property.
\end{remark}

\smallskip

The following formulas for the Killing forms of classical Lie algebra are well known:
$$
B_{so(n)}(X,Y)=-(n-2)\trace(XY),\quad  B_{sp(n)}(X,Y)=-2(n+1)\trace(XY),
$$
$$
B_{su(n)}(X,Y)=-2n\trace(XY).
$$
Hence, if we consider the inclusion $so(k+l)\subset so(k+l+m)$ with $X \mapsto \diag(X,0)=:X'$,
then
\begin{eqnarray*}
B_{so(k+l)}(X,Y)=-(k+l-2)\trace(XY),\\ B_{so(k+l+m)}(X',Y')=-(k+l+m-2)\trace(X'Y')=-(k+l+m-2)\trace(XY),
\end{eqnarray*}
and, consequently, $B_{so(k+l)}=\frac{k+l-2}{k+l+m-2}\cdot B_{so(k+l+m)}$.
Using the same argument for other two type of classical Lie algebras and Lemma \ref{constcalc}, we easily get the values of $A$, $a_1$, $a_2$, and $a_3$ for the spaces
(see Table~1):
{\small$$
SO(k+l+m)/SO(k)\cdot SO(l)\cdot SO(m),\quad Sp(k+l+m)/Sp(k)\cdot Sp(l)\cdot Sp(m), \quad SU(k+l+m)/S(U(k)\cdot U(l)\cdot U(m)).
$$}
\renewcommand{\arraystretch}{1.3}\renewcommand{\tabcolsep}{0.44cm}
\begin{table}[t]
\noindent{\bf Table 3.} The values of $\dim(\mathfrak{g})$  and $B_{\mathfrak{g}}(\beta_m,\beta_m)$ for compact simple Lie algebras $\mathfrak{g}$.\newline

\begin{center}
\begin{tabular}{|c|c|c|c|c|c|c|c|c|}
\hline
$\mathfrak{g}$& $so(n)$ & $sp(n)$ & $su(n)$ & $g_2$ & $f_4$ &\ $e_6$ &  $e_7$ & $e_8$\\
\hline
$\dim(\mathfrak{g})$ & $n(n-1)/2$ & $2n^2+n$ & $n^2-1$ & $14$ & $52$ & $78$ & $133$ & $248$ \\
\hline
$B_{\mathfrak{g}}(\beta_m,\beta_m)$ & $4(n-2)$ & $4(n+1)$ & $4n$ & $16$ & $36$ & $48$ & $72$ & $120$ \\
\hline
\end{tabular}
\end{center}
\end{table}
\renewcommand{\arraystretch}{1}

Let us consider the pair $(\mathfrak{g},\mathfrak{h})=(su(2l),u(l))$.
In this case we have $\mathfrak{k}_1=so(2l)$, $\mathfrak{k}_2=sp(l)$, and $\mathfrak{k}_3=s(u(l)\oplus u(l))$.
From the standard inclusion $so(2l)\subset su(2l)$ we get
$B_{so(2l)}(X,Y)=-(2l-2)\trace(XY)$ and $B_{su(2l)}(X,Y)=-4l\trace(XY)$, therefore, $\gamma_1=\frac{l-1}{2l}$.
By  Lemma \ref{constcalc} we get $a_1=\frac{l+1}{4l}$. Since $d_1=l(l-1)$, $d_2=l(l+1)$, and $d_3=l^2-1$, then  we get
$A=(l^2-1)/4$, $a_2=\frac{l-1}{4l}$, and $a_3=1/4$ (recall, that $a_id_i=A$).
Note that $(\tilde{\mathfrak{k}}_3,\tilde{\mathfrak{h}}_3)=\bigl(su(l)\oplus su(l),\diag(su(l))\bigr)$ in this case.
In particular, $\tilde{\mathfrak{k}}_3=su(l)\oplus su(l)$ is not a simple Lie algebra.
\smallskip

For all other cases in Table 1 we will apply  Lemma \ref{constcalc} and the following method.
Consider an inclusion $\mathfrak{k} \subset \mathfrak{g}$ of simple compact Lie algebras and try to determine a constant $\gamma$ such that
$B_{\mathfrak{k}}=\gamma \cdot B_{\mathfrak{g}}$, where $B_{\mathfrak{k}}$ and $B_{\mathfrak{g}}$ are the Killing forms of the Lie algebras
$\mathfrak{k}$ and $\mathfrak{g}$. Suppose that $\beta_m$ (respectively, $\beta'_m$) is one of the roots of maximal length in the Lie algebra
$\mathfrak{g}$ (respectively, $\mathfrak{k}$). Then the formula
\begin{equation}\label{dynk}
\gamma=\frac{B_{\mathfrak{k}}(\beta'_m, \beta'_m)}{j \cdot B_{\mathfrak{g}}(\beta_m,\beta_m)}
\end{equation}
holds, where $j$ means the Dynkin index of the Lie subalgebra $\mathfrak{k}$ in $\mathfrak{g}$,
see e.~g. \cite[ pp.~38--40]{DZ} for details. Note that the Dynkin index is a natural number and it was computed for all
simple subalgebras of exceptional Lie algebras in \cite{Dynkin1952} (see also \cite{Minchenko}).
The value $B_{\mathfrak{g}}(\beta_m,\beta_m)$ for simple Lie algebra are shown in Table 3 (this is a reproduction of Table 3 in \cite{DZ}).
\smallskip

Let us use this algorithm for the pair
$(\mathfrak{g},\mathfrak{h})=(so(2l),u(1)\oplus u(l-1))$.
In this case $\mathfrak{k}_1=su(l)\oplus \mathbb{R}$, $\mathfrak{k}_2=su(l)\oplus \mathbb{R}$, and $\mathfrak{k}_3=so(2l-2)\oplus \mathbb{R}$.
Note that $(\tilde{\mathfrak{k}}_1,\tilde{\mathfrak{h}}_1)=(\tilde{\mathfrak{k}}_2,\tilde{\mathfrak{h}}_2)=(su(l),s(u(1)\oplus u(l-1))$,
$(\tilde{\mathfrak{k}}_3,\tilde{\mathfrak{h}}_3)=(so(2l-2),u(l-1))$.
Note also that the Dynkin index $j$ for subalgebras $su(l)$ and $so(2l-2)$ in $so(2l)$ is $1$.
Using Table 3, we get $\gamma_1=\gamma_2=\frac{l}{2(l-1)}$ and $\gamma_3=\frac{l-2}{l-1}$.
Therefore, by Lemma \ref{constcalc} we have $a_1=a_2=\frac{l-2}{4(l-1)}$ and $a_3=\frac{1}{2(l-1)}$.
Since $d_1=d_2=2(l-1)$, $d_3=(l-1)(l-2)$, we also get $A=(l-2)/2$.

We list all (which will be needed) symmetric pairs $(\mathfrak{g}, \mathfrak{h})$ with exceptional $\mathfrak{g}$,
with pointing of the Dynkin index $j$  of
some simple summands ($j$ for $\mathfrak{k}$ is shown as $\mathfrak{k}^j$) in subalgebras (see \cite{Dynkin1952}):
$$
(e_6,su(6)^1\oplus su(2)),\quad  (e_6, so(10)^1\oplus \mathbb{R}),\quad  (e_6, sp(3)^1\oplus sp(1)),\quad (e_6, f_4^1),
$$
$$
(e_6, sp(4)^1),\quad (e_7, so(12)^1\oplus sp(1)),\quad (e_7, e_6^1\oplus \mathbb{R}),\quad (e_7, su(8)^1),
$$
$$
(e_8, e_7^1\oplus sp(1)),\quad (e_8, so(16)^1),\quad  (f_4, sp(3)^1\oplus sp(1)),\quad (f_4, so(9)^1),
$$
This information, together with Table 2 and Table 3, the equality (\ref{dynk}) and  Lemma \ref{constcalc}
allow us to calculate the values of $A$, $a_1$, $a_2$, and $a_3$ for all pairs in Table 1 with exceptional $\mathfrak{g}$.
\smallskip

Therefore, we get the numbers  $a_1$, $a_2$, and $a_3$ for all pairs in Table 1.

\section{The set of points $(a_1,a_2,a_3)$ in $[0,1/2]^3$}

The authors of \cite{AANS1,AANS2} studied local properties of the normalized Ricci flow for generalized Wallach spaces.
It is remarkable that the normalized Ricci flow for these space could be represented as a planar dynamical system depended in addition on
the constants $a_1$, $a_2$, and $a_3$. Even not every triple $(a_1,a_2,a_3)\in [0,1/2]^3$ corresponds to some generalized Wallach space,
it is useful to study this dynamical system with all such triples.

\begin{center}
\begin{figure}[t]
\centering\scalebox{1}[1]{
\includegraphics[angle=-90,totalheight=2.5in]{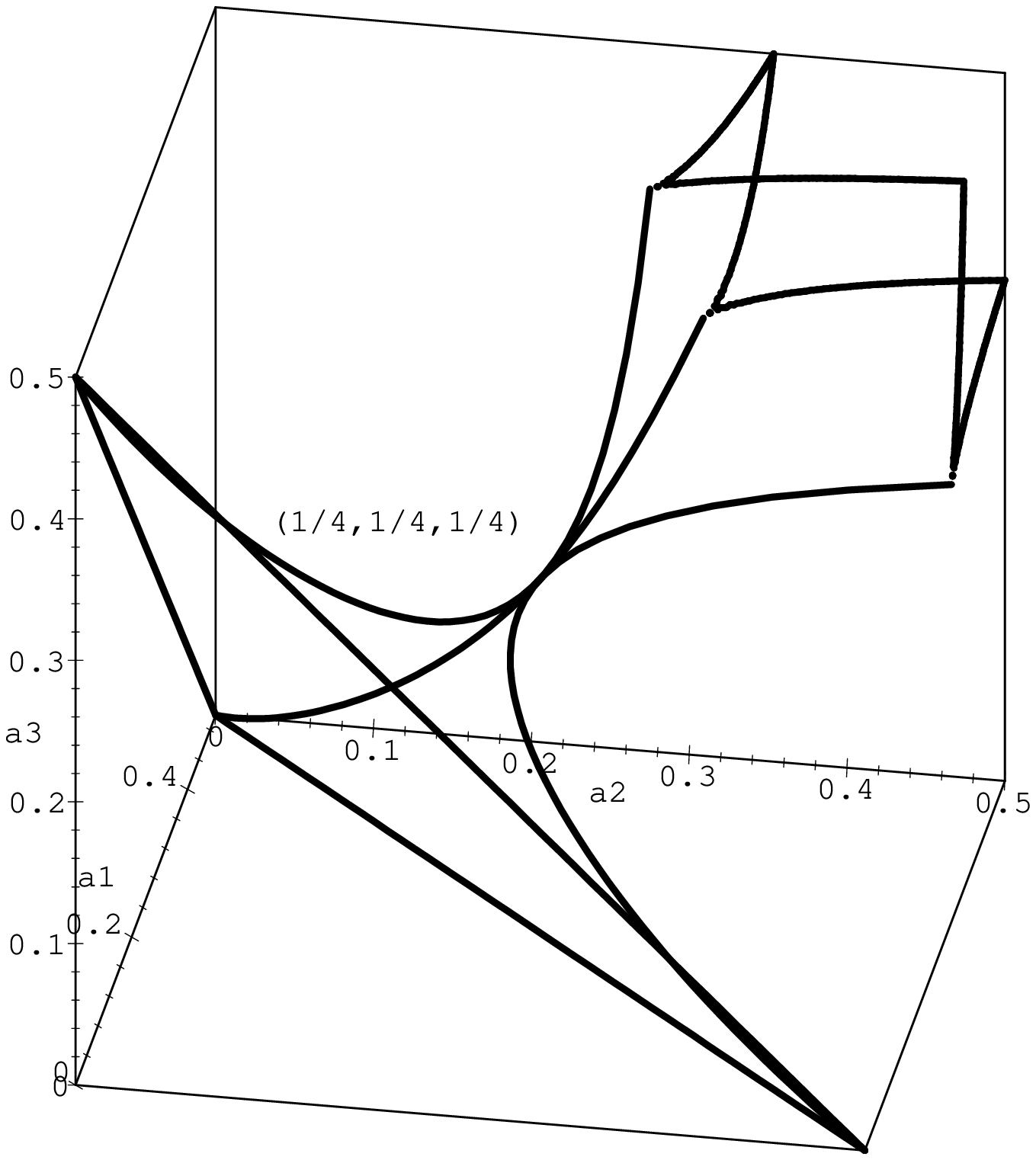}
\includegraphics[angle=-90,totalheight=2.5in]{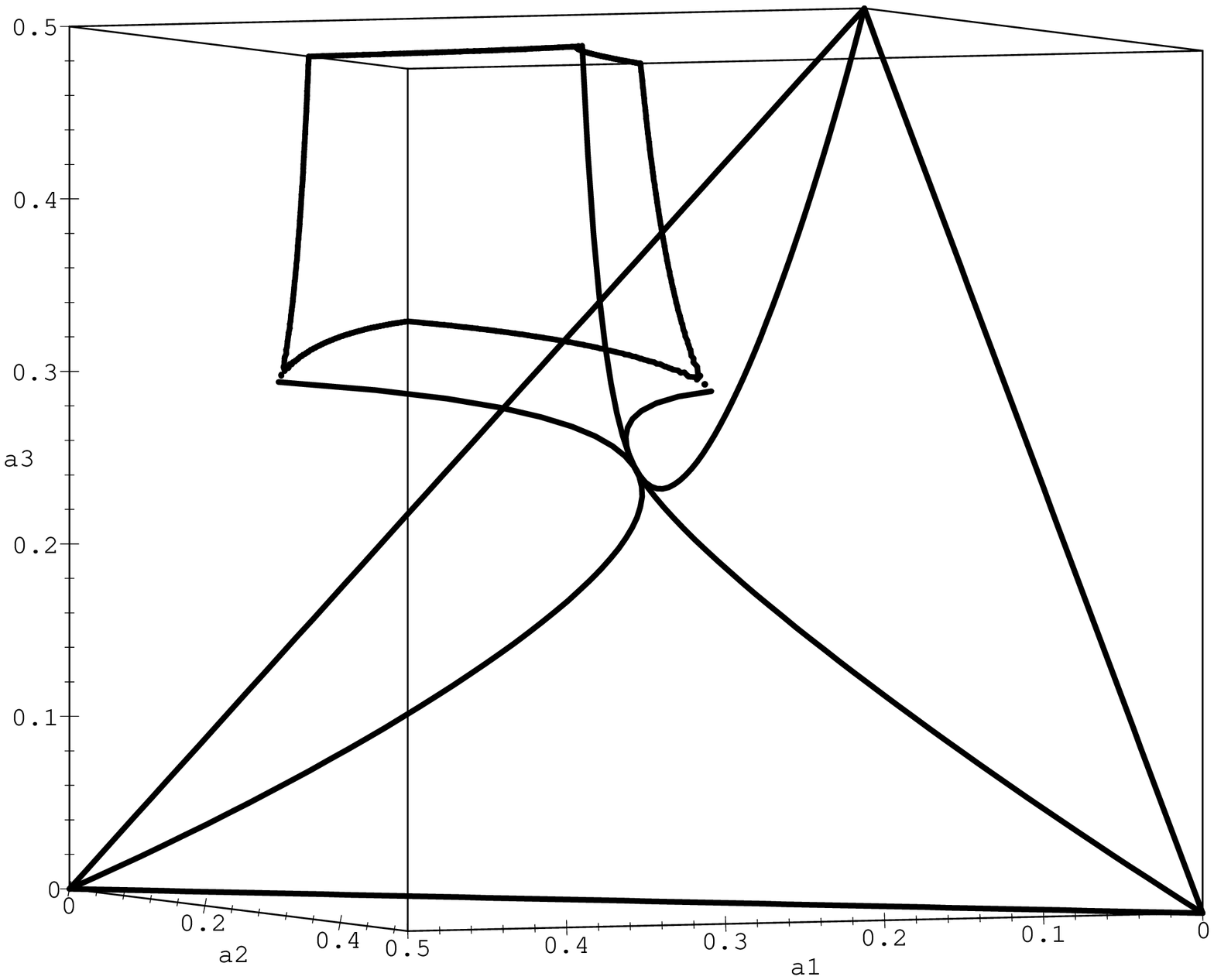}}
\caption{The surface $\Omega \cap [0,1/2]^3$.}
\label{singsur}
\end{figure}
\end{center}

Let us consider one special algebraic surface $\Omega \subset \mathbb{R}^3$, defined by the equation $Q(a_1,a_2,a_3)=0$, where
\begin{eqnarray}\label{singval2}\notag
Q(a_1,a_2,a_3)\,=\,
(2s_1+4s_3-1)(64s_1^5-64s_1^4+8s_1^3+12s_1^2-6s_1+1\\\notag
+240s_3s_1^2-240s_3s_1-1536s_3^2s_1-4096s_3^3+60s_3+768s_3^2)\\
-8s_1(2s_1+4s_3-1)(2s_1-32s_3-1)(10s_1+32s_3-5)s_2\\\notag
-16s_1^2(13-52s_1+640s_3s_1+1024s_3^2-320s_3+52s_1^2)s_2^2\\\notag
+64(2s_1-1)(2s_1-32s_3-1)s_2^3+2048s_1(2s_1-1)s_2^4,
\end{eqnarray}
and
$$
s_1 = a_1+a_2+a_3, \quad s_2 = a_1a_2+a_1a_3+a_2a_3, \quad s_3 = a_1a_2a_3.
$$

Obviously, $Q(a_1,a_2,a_3)$ is a symmetric polynomial in $a_1,a_2,a_3$ of degree 12.
The surface $\Omega$ was very important for the statement of Theorem 7 in \cite{AANS1}, which
provides a general result about the type of the non-degenerate
singular points of the normalized Ricci flow for a generalized Wallach space
with given $a_1$, $a_2$, and $a_3$.

In the rest of this section we deal only with points of the surface $\Omega$ in the cube $[0,1/2]^3$.
We recall some important properties of $\Omega$, see \cite{AANS1} for details.

The points $(0,0,1/2)$, $(0,1/2,0)$, and $(1/2,0,0)$ are all vertices of the cube $[0,1/2]^3$, that are points of $\Omega$.
For $a_1=1/2$ and $a_2,a_3 \in (0,1/2]$ points of $\Omega$ form
a curve homeomorphic to the interval $[0,1]$ with endpoints $(1/2,1/2,\sqrt{2}/4\approx 0.3535533905)$ and $(1/2,\sqrt{2}/4\approx 0.3535533905,1/2)$
and with the singular point (a cusp) at the point $a_3=a_2=(\sqrt{5}-1)/4\approx 0.3090169942$.
The same is also valid under the permutation $a_1\to a_2 \to a_3\to a_1$.

The plain $s_1=a_1+a_2+a_3=1/2$ intersects
the set $\Omega \cap [0,1/2]^3$ exactly for points in the boundary of the triangle with the vertices
$(0,0,1/2)$, $(0,1/2,0)$, and $(1/2,0,0)$. For all other points in $\Omega \cap (0,1/2]^3$ we have the inequality
$s_1=a_1+a_2+a_3>1/2$.

It is not difficult to show that $(1/4,1/4,1/4)$ is the only point in $\Omega \cap [0,1/2]^3$
satisfying the additional condition $s_1=a_1+a_2+a_3=3/4$.
It turns out that the point $(1/4,1/4,1/4)$ is a singular point of degree $3$ of the algebraic surface
$\Omega$ (see Figure \ref{singsur}). This point is an elliptic umbilic (in the sense of Darboux) on the surface $\Omega$.

Now, we discuss a part of the surface $\Omega$ in the cube $(0,1/2)^3$.
Recall that $\Omega$ is invariant under the permutation $a_1\to a_2\to a_3\to a_1$. It should be noted that the set
$(0,1/2)^3\cap \Omega$ is connected.
There are three curves (``edges'') of {\it singular points} on $\Omega$ (i.~e. points where $\nabla Q =0$):
one of them has parametric representation $a_1=-\frac{1}{2}\frac{16t^3-4t+1}{8t^2-1}, a_2=a_3=t$, and the others
are defined by permutations of $a_i$. These curves  have a common point $(1/4,1/4,1/4)$ (see Figures~\ref{singsur}). The part of $\Omega$ in $(0,1/2)^3$
consists of three (pairwise isometric) ``bubbles'' spanned on every pair of ``edges''.

Another important observation is the following:  the set $(0,1/2)^3\setminus \Omega$ has exactly three connected components.
According to \cite{AANS1}, we denote by $O_1$, $O_2$, and $O_3$ the components containing the points
$(1/6,1/6,1/6)$, $(7/15,7/15,7/15)$, and $(1/6, 1/4, 1/3)$ respectively. Note that $Q(a_1,a_2,a_3)<0$ for $(a_1,a_2,a_3) \in O_1 \cup O_2$ and
$Q(a_1,a_2,a_3)>0$ for $(a_1,a_2,a_3) \in O_3$.

It is shown in \cite{AANS1}, that the normalized Ricci flow for a generalized Wallach space with $(a_1,a_2,a_3)\in (0,1/2)^3\setminus \Omega$
has no degenerate singular point, as a planar dynamical system.
By Theorem 7 in \cite{AANS1}, for $(a_1,a_2,a_3)\in O_j$ the following possibilities for singular points of this system can occur:
\smallskip

{\bf i)} If $j=1$  then there is four singular point, one of them is an unstable node
and three other are saddles;

{\bf ii)} If $j=2$ then there is four singular point, one of them is a stable node
and three other are saddles;

{\bf iii)} If $j=3$ then there are two singular points, that are saddles.
\bigskip

Now we describe the location of points $(a_1,a_2,a_3)\in \mathbb{R}^3$ determined by generalized Wallach spaces from Theorem \ref{main}.
Recall that every such space determines not only one point $(a_1,a_2,a_3)$ but also the points that obtained with permutations of $a_1$, $a_2$, and $a_3$.
\smallskip

For the spaces $SU(k+l+m)/S\big(U(k)\times U(l) \times U(m)\big)$, $k \geq l\geq m \geq 1$, we have
$$
a_1=\frac{k}{2(k+l+m)}, \quad a_2=\frac{l}{2(k+l+m)}, \quad a_3=\frac{m}{2(k+l+m)},
$$
and $a_1+a_2+a_3=1/2$.
It is clear that all such points $(a_1,a_2,a_3)$ are in the component $O_1$.
Moreover, closure of the set of all such points coincides with the triangle in $\mathbb{R}^3$ with the
vertices $(0,0,1/2)$, $(0,1/2,0)$, and $(1/2,0,0)$. Indeed, the last assertion easily follows from considering of
the barycentric coordinates in this triangle.

For the spaces $Sp(k+l+m)/Sp(k)\times Sp(l) \times Sp(m)$, $k \geq l\geq m \geq 1$, we get
$$
a_1=\frac{k}{2(k+l+m+1)}, \quad a_2=\frac{l}{2(k+l+m+1)}, \quad a_3=\frac{m}{2(k+l+m+1)},
$$
and $a_1+a_2+a_3<1/2$. Hence, all such point are also in the component $O_1$.

The case $SO(k+l+m)/SO(k)\times SO(l) \times SO(m)$, $k \geq l\geq m \geq 1$, is more interesting. We have
$$
a_1=\frac{k}{2(k+l+m-2)}, \quad a_2=\frac{l}{2(k+l+m-2)}, \quad a_3=\frac{m}{2(k+l+m-2)}.
$$

For $l=m=1$ we get
$a_1=1/2$ and $a_2=a_3=\frac{1}{2k}$.
Hence, $(a_1,a_2,a_3) \not\in (0,1/2)^3$. Then we may assume  that $l\geq 2$ without loss of generality. Therefore, $k\geq l \geq 2$ and $k+l+m \geq 5$.

Note that $a_1+a_2+a_3 = \frac{k+l+m}{2(k+l+m-2)}=g(k+l+m)$, where $g(x)=\frac{x}{2(x-2)}$. Since the function $x \mapsto \frac{x}{2(x-2)}$
decreases for $x>2$, then we get that the inequality $a_1+a_2+a_3 \leq 3/4=g(6)$ holds for all $k,l,m$ with $k+m+l \geq 6$.

For $k+m+l \leq 5$ we should check only the space $SO(5)/SO(2)\times SO(2)\times SO(1)$ with $a_1=a_2=1/3$ and $a_3=1/6$.
It is easy to see that the point $(1/3,1/3,1/6)$ is in $O_3$, because for all points in $O_2$ we have the inequality
$a_i \geq 1/4$, $i=1,2,3$.

If  $k+m+l =6$, then $a_1+a_2+a_3 =3/4$.
Recall that the plane $a_1+a_2+a_3 =3/4$ intersects the surface $\Omega \cap (0,1/2)^3$ exactly in the point $(1/4,1/4,1/4)$ corresponded
to the space $SO(6)/SO(2)^3$. All other points of this plane in the cube $(0,1/2)^3$ are situated in the component $O_3$.
This is the case for $(a_1,a_2,a_3)=(3/8,1/4,1/8)$ corresponded to the space $SO(6)/SO(3)\times SO(2)\times SO(1)$.

For $k+m+l \geq 7$ we get $a_1+a_2+a_3 \leq 7/10=g(7)<3/4$, and such points are either in $O_1\cup O_3$ or in $\Omega$.
It is easy to see that there are infinitely many points of this type in $O_1$.
In order to find all triples in $O_1$ one should solve an inequality $F(k,l,m)<0$ for natural $k,l,m$, where $F$ is a polynomial of degree $12$.
We will not deal with this special problem here.

In any case, for the spaces $SO(k+l+m)/SO(k)\times SO(l) \times SO(m)$, there is no point $(a_1,a_2,a_3)$ in the component $O_2$.
\smallskip

Now, we determine the corresponding component $O_i$ for all other generalized Wallach spaces.
For this goal we may use all the ideas as above and one more simple observation: For $a_1=a_2=a_3=:a$, the point $(a_1,a_2,a_3)$ is in $O_1$
(respectively, $O_2$), if
$a<1/4$ (respectively, $a>1/4$).
\smallskip

Simple calculations show, that the spaces from {\bf lines 4, 7, 9}, and {\bf 15} of Table 1 are such that $(a_1,a_2,a_3)\in O_1$.
\smallskip

Further, the spaces from {\bf lines 5, 6, 8, 10, 12}, and {\bf 14} of Table 1 are such that $(a_1,a_2,a_3)\in O_3$.
Due to the first of these examples (where we have $1$\,-parameter family of spaces), we conclude that there are infinitely many points
$(a_1,a_2,a_3)$ corresponded to generalized Wallach spaces in $O_3$.
\medskip

The spaces from the lines {\bf 11} and {\bf 13} of Table 1 satisfy the condition  $(a_1,a_2,a_3)\in O_2$.
It is interesting that there are only two generalized Wallach spaces with this property.
These spaces give an affirmative answer to the question of Christoph B\"{o}hm
on the existence of specific examples of generalized Wallach spaces with the property $(a_1,a_2,a_3)\in O_2$.
\medskip

Note also that for a symmetric space $G/H$ that is a product of three irreducible symmetric space, we have $A=0$ and $(a_1,a_2,a_3)=(0,0,0)$.
Finally, for all spaces $(F \times F\times F \times F)/\diag(F)$, the equality $a_1=a_2=a_3=1/4$ holds, as well as for the space $SO(6)/SO(2)^3$.
Recall that the point $(1/4,1/4,1/4)$ is an an elliptic umbilic on the surface $\Omega$.
\medskip

\begin{remark}
When this paper was completed, the author saw the very recent preprint  \cite{CKL},
where {\bf(}in~particular{\bf)} the classification of generalized Wallach spaces $G/H$ with simple $G$
was obtained.
\end{remark}
\vspace{10mm}

\bibliographystyle{amsunsrt}

\vspace{5mm}

\end{document}